\documentclass[reqno,10pt]{amsart}
\usepackage{amsmath}
\usepackage{amssymb}
\usepackage{amsfonts}
\usepackage{hyperref}
\usepackage{graphicx}
\usepackage[abs]{overpic}
\usepackage{caption}
\usepackage{wrapfig}
\usepackage{float}
\usepackage{graphicx}

\newtheorem{theorem}{Theorem}[section]
\newtheorem{lemma}[theorem]{Lemma}
\newtheorem{proposition}[theorem]{Proposition}
\newtheorem{corollary}[theorem]{Corollary}
\newtheorem{definition}[theorem]{Definition}

\newtheorem{rem}[theorem]{Remark}
\newtheorem{example}[theorem]{Example}

\newcommand{\N}{\Bbb N}
\newcommand{\Z}{\Bbb Z}
\newcommand{\R}{\Bbb R}

\newcommand{\BBB}{\color{black}}

\newcommand{\ZZZ}{\color{black}} 
 
\newcommand{\EEE}{\color{black}} 
\newcommand{\RRR}{\color{black}}

\def\Id{\mathbf{Id}}
\def\id{\mathbf{id}}

\def\eps{\varepsilon}

\def\dist{\operatorname{dist}}

\def\XXint#1#2#3{{\setbox0=\hbox{$#1{#2#3}{\int}$}
     \vcenter{\hbox{$#2#3$}}\kern-.5\wd0}}

\usepackage{color}

\numberwithin{equation}{section}

\begin{document}

\title[Griffith energies as small strain limit of nonlinear models]{Griffith energies as small strain limit of nonlinear models for nonsimple brittle materials}


\author{Manuel Friedrich}
\address[Manuel Friedrich]{Applied Mathematics M\"unster, University of M\"unster\\
Einsteinstrasse 62, 48149 M\"unster, Germany.}
\email{manuel.friedrich@uni-muenster.de}

\begin{abstract}
We consider a nonlinear, frame indifferent Griffith model for nonsimple brittle materials where the elastic energy also depends on the second gradient of the deformations. In the framework of free discontinuity and gradient discontinuity problems, we prove existence of minimizers for boundary value problems. We then pass to a small strain limit in terms of suitably rescaled displacement fields and show that the  nonlinear energies can be identified with a  linear Griffith model in the sense of  $\Gamma$-convergence. This complements the   study in \cite{Friedrich:15-2}  by providing a linearization     result in arbitrary space dimensions.    
\end{abstract}

\maketitle

\section{Introduction}


Mathematical models in solids mechanics typically do  not predict the mechanical behavior correctly at every scale, but have a certain limited range of applicability. A central example in that direction are models for hyperelastic materials in nonlinear (finite) elasticity and their linear (infinitesimal) counterparts.  The last decades have witnessed remarkable progress in providing a clear relationship between different models via $\Gamma$-convergence \cite{DalMaso:93}. In their seminal work \cite{DalMasoNegriPercivale:02}, \ZZZ {\sc Dal Maso, Negri, and Percivale} \EEE performed a   nonlinear-to-linear analysis in terms of suitably rescaled  displacement fields and proved the convergence of minimizers for corresponding boundary value problems. This study has been extended in various directions, including  different growth assumptions on the stored energy densities \cite{Agostiniani}, the passage from atomistic-to-continuum models \cite{Braides-Solci-Vitali:07, Schmidt:2009}, multiwell energies \cite{alicandro.dalmaso.lazzaroni.palombaro, Schmidt:08}, plasticity \cite{Ulisse}, and viscoelasticity \cite{MFMK}.     

In the present contribution, we are interested in an analogous analysis for materials undergoing fracture. Based on  the variational approach to quasistatic crack evolution by {\sc Francfort and Marigo} \cite{Francfort-Marigo:1998}, where the displacements and the   (a priori unknown) crack paths are determined from an energy minimization principle,  we consider an energy functional of Griffith-type. Such   variational models of brittle fracture,   which comprise an elastic energy stored in the uncracked region of the body and a surface  contribution comparable to the size of the crack of codimension one, have been  widely studied both at finite and infinitesimal strains, see \cite{BJ, Chambolle:2003, DalMaso-Francfort-Toader:2005,  DalMaso-Lazzaroni:2010, Francfort-Larsen:2003, Solombrino, Giacomini-Ponsiglione:2006} without claim of being exhaustive. We refer the reader to \cite{Bourdin-Francfort-Marigo:2008} for  a general overview. 

In this context, first results addressing  the question of a nonlinear-to-linear analysis have been obtained in \cite{NegriToader:2013, Zanini} in a two-dimensional evolutionary setting for a fixed crack set or a restricted class of admissible cracks, respectively. Subsequently, the problem was studied in \cite{FriedrichSchmidt:2014.2} from a different perspective. Here, a simultaneous discrete-to-continuum and nonlinear-to-linear analysis is performed for general crack geometries, but under the simplifying assumption that all deformations \BBB are \EEE close to the identity mapping.

Eventually, a result in dimension two without a priori assumptions on the crack paths and the deformations, in the general framework of free discontinuity problems (see \cite{DeGiorgi-Ambrosio:1988}), has been derived  in \cite{Friedrich:15-2}. \BBB This \EEE analysis relies fundamentally on delicate geometric \BBB rigidity \EEE  results in the spirit of \cite{FrieseckeJamesMueller:02, Chambolle-Giacomini-Ponsiglione:2007}. At this point, the geometry of crack paths in the plane is crucially exploited  and higher dimensional analogs seem to be currently out of reach. In spite of the lack of rigidity estimates, the goal of this contribution is to perform a nonlinear-to-linear analysis for brittle materials in the spirit of   \cite{Friedrich:15-2} in higher space dimensions. This will be achieved by starting from a slightly different nonlinear model for so-called \emph{nonsimple materials}.

Whereas the elastic properties of simple   materials depend only on the first gradient, the notion of a  nonsimple material \EEE refers to the fact that the elastic energy depends additionally on the second gradient of the deformation. This idea goes back to {\sc Toupin} \cite{Toupin:62,Toupin:64} and has proved to  be useful in modern mathematical elasticity, see e.g.~\cite{BCO,Batra, capriz,dunn, MFMK,MR}, since it brings additional compactness and rigidity to the problem. In a similar fashion, we consider here a Griffith model with an additional second gradient in the elastic part of the energy. This leads to a model in the framework of free discontinuity and gradient discontinuity problems.

The goal of this contribution is twofold. We first show that the regularization allows to prove existence of minimizers for boundary value problems without convexity properties for the stored elastic energy. In particular, we do not have to assume quasiconvexity \cite{Ambrosio:90-2}. Afterwards,  we identify an effective linearized Griffith energy as the $\Gamma$-limit of the nonlinear and frame indifferent models for vanishing strains. In this context, it is important to mention that, in spite of the formulation of the nonlinear model in terms of nonsimple materials, the effective limit is a   `standard' Griffith functional in linearized elasticity depending only on the first gradient. A similar justification for the treatment of nonsimple materials has recently been discussed in \cite{MFMK} for a model in nonlinear viscoelasticity.

The existence result for boundary value problems at finite strains is formulated in the space $GSBV^2_2(\Omega;\R^d)$, see \eqref{eq: space} below, consisting of the mappings for which  both the function itself and its derivative are in the class of \emph{generalized special functions of bounded variation} \cite{Ambrosio-Fusco-Pallara:2000}. The relevant compactness and lower semicontinuity results stated in Theorem \ref{th: comp} essentially follow from a study on second order variational problems with free discontinuity and gradient discontinuity \cite{Carriero2}. \BBB Another \EEE key ingredient is the   recent work \cite{Manuel} which extends the classical compactness result due to {\sc Ambrosio} \cite{Ambrosio:90} to  problems without a priori bounds on the functions.

Concerning the passage to the linearized system, the essential step is to establish a compactness result in terms of suitably rescaled  displacement fields \BBB which  measure \EEE  the distance of the deformations from the identity. Whereas in \cite{Friedrich:15-2} this is achieved  by means of delicate geometric rigidity estimates, the main idea in our approach is to partition the domain into different regions in which the gradient is `almost constant'. This construction relies on the coarea formula in $BV$ and  is the fundamental point where the presence of a second order term \BBB in the energy \EEE is used to pass rigorously to a linear theory. The linear limiting model is formulated on the space of \emph{generalized special functions of bounded deformation} $GSBD^2$, which has been studied extensively over the last years, see e.g.\  \cite{Chambolle-Conti-Francfort:2014, Chambolle-Conti-Francfort:2018, Chambolle-Conti-Iurlano:2018,  Conti-Iurlano:15, Conti-Iurlano:15.2, Crismale2, Crismale, Crismale4, Crismale3, DalMaso:13, Friedrich:15-3, Friedrich:15-4, Solombrino, Iurlano:13}.

The paper is organized as follows. In Section \ref{rig-sec: main} we first introduce our nonlinear model for nonsimple brittle materials and state our main results: we first address the existence of minimizers for boundary value problems at finite strains. Then, we present a compactness and $\Gamma$-convergence result in the passage from \BBB the nonlinear to the linearized \EEE theory. Here, we also discuss the convergence of minima and minimzers \BBB under given boundary data. \EEE  Section \ref{sec:pre} is devoted to some preliminary results about the function spaces $GSBV$ and $GSBD$. In particular, we present a compactness result in $GSBV^2_2$ involving the second gradient (see Theorem \ref{th: comp}). Finally, Section \ref{sec: proofs} contains the proofs of our results.

\section{The model and main results}\label{rig-sec: main}

In this section we introduce our model and present the main results. We start with some basic notation.  Throughout the paper, $\Omega \subset \R^d$ is an open and  bounded set.  The notations $\mathcal{L}^d$ and $\mathcal{H}^{d-1}$ are used for  the Lebesgue measure and the $(d-1)$-dimensional Hausdorff measure in $\mathbb{R}^d$, respectively. We set  $S^{d-1}=\lbrace x \in \R^d: \, |x|=1\rbrace$. For an $\mathcal{L}^d$-measurable set $E\subset\mathbb{R}^d$,  the symbol $\chi_E$ denotes its indicator function.  For two sets $A,B \subset \R^d$, we \BBB define \EEE $A \triangle B = (A\setminus B) \cup (B \setminus A)$.  The identity mapping on $\R^d$ is indicated by $\id$ and its derivative, the identity matrix, by $\Id \in \R^{d \times d}$.  The sets of symmetric and skew symmetric matrices are denoted by $\R^{d \times d}_{\rm sym}$ and $\R^{d \times d}_{\rm skew}$, respectively. We set ${\rm sym}(F) = \frac{1}{2}(F^T + F)$ for $F \in \R^{d\times d}$ \BBB and \EEE define $SO(d) = \lbrace R\in \R^{d \times d}: R^T R = \Id, \, \det R=1 \rbrace$.

\subsection{A nonlinear model for nonsimple materials and boundary value problems}

In this subsection we introduce our nonlinear model and discuss the existence of minimizers for boundary value problems.

\emph{Function spaces:} To introduce our Griffith-type model for nonsimple materials, we first need to introduce the relevant spaces. We use standard notation for $GSBV$ functions, see \cite[Section 4]{Ambrosio-Fusco-Pallara:2000} and \cite[Section 2]{DalMaso-Francfort-Toader:2005}. In particular, we let
\begin{align}\label{eq: space2}
GSBV^2(\Omega;\R^d) = \lbrace y \in GSBV(\Omega;\R^d): \ \nabla y \in L^2(\Omega;\R^{d\times  d}), \ \mathcal{H}^{d-1}(J_y) < + \infty \rbrace,
\end{align}
where $\nabla y(x)$ denotes the approximate differential at $\mathcal{L}^d$-a.e.\ $x \in \Omega$  and $J_y$ the jump set.  We define the space
\begin{align}\label{eq: space}
GSBV^2_2(\Omega;\R^d) := \big\{ y \in GSBV^2(\Omega; \R^d): \ \nabla y \in GSBV^2(\Omega;\R^{d\times d})\big\}.
\end{align}
The approximate differential and the jump set of $\nabla y$ will be denoted by $\nabla^2 y$ and $J_{\nabla y}$, respectively. \ZZZ (To avoid confusion, we point out that in the paper   \cite{DalMaso-Francfort-Toader:2005} the notation $GSBV^2_2(\Omega;\R^d)$ was used differently, namely for $GSBV^2(\Omega;\R^d) \cap L^2(\Omega;\R^d)$.) \EEE 

A similar space  has been considered in \cite{Carriero1, Carriero2} to treat second order free discontinuity functionals, e.g., a weak formulation 
of the Blake $\&$ Zissermann model \cite{BZ} of image segmentation. We point out  that the functions are allowed to exhibit \ZZZ discontinuities. \EEE Thus, the analysis is outside of the framework of the space of special functions with bounded Hessian $SBH(\Omega)$, considered in problems of second order energies for elastic-perfectly plastic plates, see e.g.\ \cite{Carriero3}.

\emph{Nonlinear Griffith energy for nonsimple materials:}   We let $W:\R^{d \times d} \to [0,+\infty)$ be a  single well, frame indifferent stored energy functional. \BBB More precisely, \EEE  we suppose  that there exists $c>0$ such that
\begin{align}\label{assumptions-W}
{\rm (i)} & \ \ W \text{ continuous and $C^3$ in a neighborhood of $SO(d)$},\notag\\
{\rm (ii)} & \ \ \text{Frame indifference: } W(RF) = W(F) \text{ for all } F \in \R^{d \times d}, R \in SO(d),\notag\\
{\rm (iii)} & \ \ W(F) \ge c\dist^2(F,SO(d)) \ \text{ for all $F \in \R^{d \times d}$}, \  W(F) = 0 \text{ iff } F \in SO(d).
\end{align}
We briefly note that we can also treat inhomogeneous materials where the energy density has the form $W: \Omega \times \R^{d \times d} \to [0,+\infty)$. Moreover, it suffices to assume $W \in C^{2,\alpha}$, where $C^{2,\alpha}$ is the H\"older space with exponent $\alpha \in (0,1]$, see Remark \ref{rem: Hoelder space} for details.

Let $\kappa>0$ and  $\beta \in (\frac{2}{3},1)$. \EEE For  $\eps >0$, define the energy $\mathcal{E}_\eps(\cdot,\Omega) : GSBV^2_2(\Omega;\R^d) \to [0,+\infty]$ by
\begin{align}\label{rig-eq: Griffith en}
\mathcal{E}_\eps(y,\Omega) =  \begin{cases} \eps^{-2}\int_{\Omega} W(\nabla y(x)) \,dx +\eps^{-2\beta} \int_\Omega |\nabla^2 y(x)|^2 \, dx  +  \kappa\mathcal{H}^{d-1}(J_y)   & \text{ if $J_{\nabla y} \subset J_y$,} \\ + \infty & \text{ else.} \end{cases}
\end{align}
Here and in the following, the inclusion  $J_{\nabla y} \subset J_y$ has to be understood up to an $\mathcal{H}^{d-1}$-negligible set.  Since $W$ grows quadratically around $SO(d)$, the parameter $\eps$ corresponds to the typical scaling of strains for configurations with finite energy. 

Due to the presence of the second term, we deal with a Griffith-type \BBB model \EEE  for \emph{nonsimple materials}. As explained in the introduction,   elastic energies which depend additionally on the second gradient of the deformation were introduced by {\sc Toupin} \cite{Toupin:62,Toupin:64} to \BBB  enhance    compactness and rigidity properties. \EEE  In the present context, we add  a second gradient term for a material undergoing fracture. This regularization effect  acts on the \RRR entire \EEE \emph{intact region} $\Omega \setminus \ZZZ J_y \EEE $ of the material. This is modeled by the condition $J_{\nabla y} \subset J_y$. 

The goal of this contribution is twofold. We first show that the regularization allows to prove existence of minimizers for boundary value problems without convexity   properties of $W$.  The main result of the present work is then to identify  a  linearized Griffith energy in the small strain limit $\eps \to 0$  which is related to the nonlinear energies $\mathcal{E}_\eps$ through $\Gamma$-convergence. We point out that the effective limit is a   `standard' Griffith model in linearized elasticity depending only on the first gradient,   see \eqref{rig-eq: Griffith en-lim} below,   although we start \BBB with \EEE a nonlinear model for nonsimple materials.

 We observe  that the condition $J_{\nabla y} \subset J_y$ is not closed under convergence in measure on $\Omega$. \ZZZ In fact, consider, e.g., $\Omega = (-1,1)^2, \Omega_1 = (-1,0)\times (-1,1), \Omega_2=(0,1)\times (-1,1)$, and for $\delta \ge 0$ the configurations
 $$y_\delta(x_1,x_2) = (x_1,x_2) \chi_{\Omega_1} + (2x_1 + \delta,x_2)\chi_{\Omega_2} \ \ \ \ \ \ \text{for } (x_1,x_2) \in \Omega.    $$
 Then $J_{\nabla y_\delta} = J_{y_\delta} = \lbrace 0\rbrace \times (-1,1)$    for $\delta>0$ and $y_\delta \to y_0$ in measure on $\Omega$ as $\delta \to 0$. However, there holds $\emptyset = J_{y_0} \subset J_{\nabla y_0} = \lbrace 0\rbrace \times (-1,1)$. \EEE   Therefore, we need to pass to a relaxed formulation.

\begin{proposition}[Relaxation]\label{prop: relaxation}
Let $\Omega \subset \R^d$ be open and bounded. Suppose that $W$ satisfies \eqref{assumptions-W}. Then the relaxed functional $\overline{\mathcal{E}}_\eps(\cdot,\Omega): GSBV^2_2(\Omega;\R^d) \to [0,+\infty] $ defined by 
$$\overline{\mathcal{E}}_\eps(y,\Omega) = \inf \big\{ \liminf\nolimits_{n \to \infty} \mathcal{E}_\eps(y_n,\Omega): y_n \to y \ \text{\rm  in measure on $\Omega$} \big\}$$ is given by
\begin{align}\label{eq: relaxed energy}
\overline{\mathcal{E}}_\eps(y,\Omega) =  \eps^{-2}\int_{\Omega} W(\nabla y(x)) \,dx +\eps^{-2\beta} \int_\Omega |\nabla^2 y(x)|^2 \, dx  +  \kappa\mathcal{H}^{d-1}(J_y \cup J_{\nabla y}).
\end{align}
\end{proposition}
The result is proved in Subsection \ref{sec: proofs-nonlinear}. Clearly, $\overline{\mathcal{E}}_\eps(\cdot,\Omega)$ is  lower semicontinuous with respect to the convergence in measure. \RRR We point out that this latter property \EEE has essentially been shown in \cite{Carriero2}, cf.\ Theorem \ref{th: GSBV2 comp}.

In the following, our goal is to study boundary value problems. To this end, we suppose that there exist two bounded Lipschitz domains $\Omega' \supset \Omega$. We will impose Dirichlet boundary data on  $\partial_D \Omega := \Omega' \cap \partial \Omega$. As usual for the weak formulation in the framework of free discontinuity problems, this will be done by requiring that configurations $y$ satisfy $y = g$ on $\Omega' \setminus \overline{\Omega}$ for some  $g \in W^{2,\infty}(\Omega';\R^d)$.   From  now on, we write $\mathcal{E}_\eps(\cdot) = \mathcal{E}_\eps(\cdot,\Omega')$ and $\overline{\mathcal{E}}_\eps(\cdot) = \overline{\mathcal{E}}_\eps(\cdot,\Omega')$ for notational convenience.  The following result about existence of minimizers will be proved in Subsection \ref{sec: proofs-nonlinear}.

\begin{theorem}[Existence of minimizers]\label{th: existence}
Let $\Omega \subset \Omega' \subset \R^d$ be bounded Lipschitz domains. Suppose that $W$ satisfies \eqref{assumptions-W}, \BBB and let \EEE  $g \in  W^{2,\infty}(\Omega';\R^d)$. Then the minimization problem 
\begin{align}\label{eq: minimization problem}
\inf_{{y \in GSBV^2_2(\Omega';\R^d)}} \Big\{ \overline{\mathcal{E}}_\eps(y): \  y = g \text{ on } \Omega' \setminus \overline{\Omega} \Big\}
\end{align}
admits solutions. 
\end{theorem}

\subsection{Compactness of rescaled displacement fields}

The main goal of the present work is the identification of an effective linearized Griffith energy in the small strain limit. In this subsection, we formulate  the relevant compactness result. Let $\Omega' \supset \Omega$ be bounded Lipschitz domains.  The limiting energy is defined on the space of generalized special functions of bounded deformation $GSBD^2(\Omega')$. For basic properties of $GSBD^2(\Omega')$ we refer to \cite{DalMaso:13} and  Section \ref{sec: GSBD} below. In particular, for $u\in GSBD^2(\Omega')$, we denote by $e(u) = \frac{1}{2}(\nabla u^T + \nabla u )$  the approximate symmetric differential and by $J_u$ the jump set.

\RRR The general idea in linearization results \EEE in many different settings (see, e.g., \cite{alicandro.dalmaso.lazzaroni.palombaro, Braides-Solci-Vitali:07, DalMasoNegriPercivale:02, MFMK, FriedrichSchmidt:2014.2, NegriToader:2013, Schmidt:08, Schmidt:2009})  is the following: given a sequence $(y_\eps)_\eps$ with $\sup_\eps \mathcal{E}_\eps(y_\eps) < +\infty$, define  displacement fields which measure the distance of the deformations from the identity, rescaled by the \ZZZ small parameter \EEE $\eps$, i.e.,
\begin{align}\label{eq: rescali1}
u_\eps = \frac{1}{\eps}(y_\eps - \id).
\end{align}
It turns out, however, that in general no compactness can be expected if the body may undergo fracture. Consider, e.g., the functions $y_\eps = \id \chi_{\Omega' \setminus B} + R\,x \chi_{B}$, for a small ball  $B \subset \Omega$ and a rotation $R \in SO(d)$, $R \neq \Id$. Then $|u_\eps|, |\nabla u_\eps| \to \infty$ on $B$ as $\eps \to 0$. The main idea in our approach is the observation that this phenomenon can be avoided  if the deformation is \emph{rotated back to the identity} on the set $B$. This will be made precise  in Theorem \ref{thm: compactess}(a) below where we pass to \emph{piecewise rotated functions}. For such functions, we can control at least the symmetric part  of $\nabla u_\eps$ for the rescaled displacement fields defined \BBB in \EEE \eqref{eq: rescali1}. This will allow us to derive a compactness result in the space $GSBD^2(\Omega')$, see Theorem \ref{thm: compactess}(b).

Recall the definition of $GSBV_2^2(\Omega';\R^d)$ in \eqref{eq: space}. To account for boundary data $h \in W^{2,\infty}(\Omega';\R^d)$, we introduce the spaces
\begin{align}\label{eq: boundary-spaces}
\mathcal{S}_{\eps,h} & = \lbrace y \in GSBV_2^2(\Omega';\R^d): \ y = \id + \eps h \text{ on } \Omega' \setminus \overline{\Omega} \rbrace, \notag\\
GSBD^2_h & = \lbrace u \in GSBD^2(\Omega'): \, u = h \text{ on } \Omega' \setminus \overline{\Omega} \rbrace.
\end{align}
Recall  $\beta \in (\frac{2}{3},1)$    and the definition of $\overline{\mathcal{E}}_\eps = \overline{\mathcal{E}}_\eps(\cdot,\Omega')$ in \eqref{eq: relaxed energy}.
 For definition and basic properties of Caccioppoli partitions we refer to Section \ref{sec: Caccio}.   In particular, for a set of finite perimeter $E \subset \Omega'$, we denote by $\partial^* E$ its essential boundary \BBB and by  $(E)^1$ the points where $E$ has density one, see \cite[Definition 3.60]{Ambrosio-Fusco-Pallara:2000}. \EEE

\begin{theorem}[Compactness]\label{thm: compactess}
Let $\gamma \in (\frac{2}{3},\beta)$. Assume that $W$ satisfies \eqref{assumptions-W}, \BBB and let \EEE  $h \in W^{2,\infty}(\Omega';\R^d)$. Let $(y_\eps)_\eps $ be a sequence satisfying $y_\eps \in \mathcal{S}_{\eps,h}$  and $\sup_\eps \overline{\mathcal{E}}_\eps(y_\eps) <+\infty$. 

\noindent (a) (Piecewise rotated functions) There exist Caccioppoli partitions $(P_j^\eps)_j$ of $\Omega'$ and corresponding rotations $(R^\eps_j)_j \subset SO(d)$ such that the piecewise  rotated functions $y^{\rm rot}_\eps \in GSBV^2_2(\Omega';\R^d)$ given by
\begin{align}\label{eq: piecewise rotated}
y^{\rm rot}_\eps    :=      \sum\nolimits_{j=1}^\infty R^\eps_j \,  y_\eps \,  \chi_{P^\eps_j} 
\end{align}
satisfy
\begin{align}\label{eq: first conditions}
{\rm (i)} & \ \ \text{$y^{\rm rot}_\eps = \id + \eps h$ on $\Omega' \setminus \overline{\Omega}$,}\notag \\ 
{\rm (ii)} & \ \  \mathcal{H}^{d-1}\big( \big(J_{y_\eps^{\rm rot}} \cup  J_{\nabla y_\eps^{\rm rot}} \big)  \setminus \big( J_{y_\eps} \cup J_{\nabla y_\eps}\big)\big) \le   \mathcal{H}^{d-1}\Big( \Big( \Omega' \cap \bigcup\nolimits_{j=1}^\infty \partial^* P_j^\eps \Big)\setminus J_{  \nabla y_\eps} \Big) \le C\eps^{\beta-\gamma}, \notag \\
{\rm (iii)}  & \ \    \Vert {\rm sym} (\nabla y_\eps^{\rm rot}) -\Id\Vert_{L^2(\Omega')} \le C\eps,\notag \\ 
{\rm (iv)}  & \ \    \Vert  \nabla y_\eps^{\rm rot} - \Id \Vert_{L^2(\Omega')} \le C\eps^\gamma
\end{align}
for a constant $C>0$ independent of $\eps$.

\noindent (b)(Compactness of rescaled displacement fields) There exists a subsequence (not relabeled) and  a function $u \in GSBD^2_h$   such that the rescaled displacement fields $u_\eps \in GSBV^2_2(\Omega';\R^d)$ defined by
\begin{align}\label{eq: rescalidipl}
u_\eps:  = \frac{1}{\eps} (y^{\rm rot}_\eps - \id) 
\end{align}     
satisfy
\begin{align}\label{eq: the main convergence}
{\rm (i)} & \ \ u_\eps \to u \ \ \ \text{a.e.\ in $\Omega' \setminus E_u$},\notag \\
{\rm (ii)} & \ \ e(u_\eps) \rightharpoonup e(u)  \ \ \text{weakly in $L^2(\Omega' \setminus E_u;\R^{d\times d}_{\rm sym})$},\notag\\
 {\rm (iii)} & \ \   \mathcal{H}^{d-1}(J_{u}) \le \liminf\nolimits_{\eps \to 0}  \mathcal{H}^{d-1}(J_{u_\eps}) \le \liminf\nolimits_{\eps \to 0}  \mathcal{H}^{d-1}(J_{y_\eps} \cup J_{\nabla y_\eps}),\notag\\
 {\rm (iv)} & \ \ e(u) = 0 \ \ \text{ on } \ E_u, \ \ \ \mathcal{H}^{d-1}\big((\partial^* E_u \BBB \cap \Omega' \EEE )\setminus J_u \big) = \mathcal{H}^{d-1}(J_u \cap \BBB (E_u)^1 \EEE ) = 0,
\end{align}
where $E_u := \lbrace x\in \Omega: \, |u_\eps(x)| \to \infty \rbrace$ is a set of finite perimeter. 
 \end{theorem}

Here and in the sequel, we follow the usual convention that convergence of the continuous parameter $\eps \to 0$ stands for convergence of arbitrary sequences $\lbrace \eps_i \rbrace_i$ with $\eps_i \to 0$ as $i \to \infty$, see \cite[Definition 1.45]{Braides:02}.  The compactness result  will be proved in Subsection \ref{sec: comp-proof}.

Note that \eqref{eq: first conditions}(i) implies $y_\eps^{\rm rot} \in \mathcal{S}_{\eps,h}$. In view of  \eqref{eq: first conditions}(ii), the frame indifference of the elastic energy, and $\gamma <\beta$, \BBB one can show that \EEE the Griffith-type energy \eqref{eq: relaxed energy} of $y_\eps^{\rm rot}$ is \BBB asymptotically   not larger than the one of $y_\eps$. \EEE  The control on the symmetric part of the derivative \eqref{eq: first conditions}(iii) is essential to obtain compactness in $GSBD^2(\Omega')$ for the sequence $(u_\eps)_\eps$. Property \eqref{eq: first conditions}(iv) will be needed to control higher order terms in the passage to linearized elastic energies, see Theorem \ref{rig-th: gammaconv} below.

The presence of the set $E_u$ is due to the compactness result in $GSBD^2(\Omega')$, see  \cite{Crismale} and  Theorem \ref{th: crismale-comp}. In principle, the phenomenon that the sequence is unbounded on a set of positive measure can be avoided by generalizing the definition of \eqref{eq: rescalidipl}: in \cite[Theorem 6.1]{Solombrino} and \cite[Theorem 2.2]{Friedrich:15-2} it has been shown that, by subtracting in \eqref{eq: rescalidipl} suitable translations on a Caccioppoli partition of $\Omega'$ related to $y_\eps$, one can achieve $E_u = \emptyset$. This construction, however, is limited \BBB so far   to dimension two. \EEE  As discussed in \cite{Crismale}, the presence of $E_u$ is not an issue  for minimization problems of Griffith energies since a minimizer can be recovered by choosing $u$ affine on $E_u$ with $e(u)=0$, cf.\ \eqref{eq: the main convergence}(iv). We also note that $E_u \subset \Omega$, i.e., $E_u \cap (\Omega'\setminus \overline{\Omega}) = \emptyset$.

\begin{definition}[Asymptotic representation]\label{def:conv}
{\normalfont
We say that a sequence $(y_\eps)_\eps$ with $y_\eps \in \mathcal{S}_{\eps,h}$ is \emph{asymptotically represented} by a limiting displacement $u \in GSBD^2_h$, and write $y_\eps \rightsquigarrow u$, if there exist sequences of Caccioppoli partitions $(P_j^\eps)_j$ of $\Omega'$ and corresponding rotations $(R^\eps_j)_j \subset SO(d)$ such that  \eqref{eq: first conditions} and \eqref{eq: the main convergence} hold for some fixed $\gamma \in (\frac{2}{3},\beta)$, where $y_\eps^{\rm rot}$ and $u_\eps$ are defined in \eqref{eq: piecewise rotated} and \eqref{eq: rescalidipl}, respectively. 
}
\end{definition}
Theorem \ref{thm: compactess} shows that for each $(y_\eps)_\eps$ with $\sup_\eps \overline{\mathcal{E}}_\eps(y_\eps) < + \infty$ there exists a subsequence $(y_{\eps_k})_k$ and $u \in GSBD^2_h$ such that $y_{\eps_k}  \rightsquigarrow u$ as $k \to \infty$.
We speak of asymptotic representation instead of convergence, and we use the symbol $ \rightsquigarrow $, in order to emphasize that  Definition \ref{def:conv} cannot be understood as a convergence with respect to a certain topology. In particular, the  limiting function $u$ for a given (sub-)sequence $(y_\eps)_\eps$ is not determined uniquely, but  depends fundamentally on the choice of the sequences $(P_j^\eps)_j$ and $(R^\eps_j)_j$. To illustrate this phenomenon, we consider an example similar to \cite[Example 2.4]{Friedrich:15-2}.

\begin{example}[Nonuniqueness of limits]\label{ex}
{\normalfont
Consider $\Omega'= (0,3) \times (0,1)$, $\Omega = (1,3) \times (0,1)$, $\Omega_1 = (0,2) \times (0,1)$, $\Omega_2 = (2,3) \times (0,1)$, $h \equiv 0$, and   
$$y_\eps(x) = x \,\chi_{\Omega_1}(x) + \bar{R}_{\eps}\,x \,\chi_{\Omega_2}(x) \ \ \  \ \text{for} \ x \in \Omega',$$
where $\bar{R}_\eps \in SO(2)$ with $\bar{R}_\eps = \Id + \eps A + {\rm O}(\eps^2) $ for some $A \in \R^{2 \times 2}_{\rm skew}$.    Then two possible alternatives are
\begin{align*}
(1)& \ \ P^{\eps}_1 = \Omega_1, \  \ P^{\eps}_2 = \Omega_2, \ \  R_1^{\eps} = \Id, \ \ R_2^{\eps} = \bar{R}_\eps^{-1},\\
(2)& \ \ \tilde{P}^{\eps}_1 = \Omega',  \ \ \tilde{R}_1^{\eps} = \Id.
\end{align*}
Letting $u_{\eps} = \eps^{-1} (\sum_{j=1}^2 R_j^{\eps}y_{\eps} \chi_{P_j^{\eps}} -\id)$ and $\tilde{u}_{\eps} = \eps^{-1}  (y_{\eps} -\id)$, we find  the limits $u \equiv 0$ and $\tilde{u}(x) = A\,x \,  \chi_{\Omega_2}(x)$, respectively. }
\end{example}

We refer to \cite[Section 2.3]{Friedrich:15-2} for a further discussion about different choices of the involved partitions and rigid motions. Here, we show that it is possible to identify uniquely the relevant notions $e(u)$ and $J_u$ of the limit. This is content of the following lemma.

\begin{lemma}[Characterization of limiting displacements]\label{lemma: characteri}
Suppose that   a sequence $(y_\eps)_\eps$ satisfies $y_\eps  \rightsquigarrow  u_1$ and $y_\eps  \rightsquigarrow  u_2$, where $u_1, u_2 \in GSBD^2_h$, $u_1 \neq u_2$. Let $E_{u_1}, E_{u_2} \subset \Omega$ be the sets given in \eqref{eq: the main convergence}. Then 

\begin{itemize}
\item[(a)]  $e(u_1) = e(u_2)$ $\mathcal{L}^d$-a.e.\ on $\Omega' \setminus (E_{u_1} \cup E_{u_2})$.
\item[(b)]  If additionally $(y_\eps)_\eps$ is a   minimizing sequence, i.e.,
\begin{align}\label{eq: minimizer}
\overline{\mathcal{E}}_\eps(y_\eps) \le \inf_{\bar{y} \in \mathcal{S}_{\eps,h}} \overline{\mathcal{E}}_\eps(\bar{y}) + \rho_\eps \  \ \ \ \ \text{with } \ \rho_\eps \to 0 \text{ as $\eps \to 0$},   
\end{align}
then $e(u_1) = e(u_2)$ $\mathcal{L}^d$-a.e.\ on $\Omega'$, and $J_{u_1} = J_{u_2}$ up to an $\mathcal{H}^{d-1}$-negligible set. 
\end{itemize}
\end{lemma}

Note that property (a) is consistent with Example \ref{ex}. Example \ref{ex} also shows that the property $J_{u_1} = J_{u_2}$ is not satisfied in general but some extra condition, e.g.\ the one in \eqref{eq: minimizer}, is necessary. We refer to Example \ref{e2} below for an \BBB illustration \EEE that in case (a)  the strains are not necessarily the same inside $E_{u_1} \cup E_{u_2}$. The result will be proved in Subsection \ref{sec: admissible}.

\subsection{Passage from the nonlinear to a linearized Griffith model}\label{rig-sec: sub, main gamma}

We now show that the nonlinear energies of Griffith-type can be related to a linearized Griffith model in the small strain limit  by $\Gamma$-convergence. We also discuss the convergence of minimizers for boundary value problems. Given bounded Lipschitz domains $\Omega \subset \Omega'$, we define the energy $\mathcal{E}: GSBD^2(\Omega') \to [0,+\infty)$ by 
\begin{align}\label{rig-eq: Griffith en-lim}
\mathcal{E}(u) =  \int_{\Omega'} \frac{1}{2} Q(e(u))  + \kappa\mathcal{H}^{d-1}(J_u),
\end{align}
where  $\kappa>0$, and  $Q: \R^{d \times d} \to [0,+\infty)$ is the quadratic form $Q(F) = D^2W(\Id)F : F$ for all $F \in \R^{d \times d}$. In view of  \eqref{assumptions-W},  $Q$ is positive definite on $\R^{d \times d}_{\rm sym}$ and vanishes on $\R^{d \times d }_{\rm skew}$. 

For the $\Gamma$-limsup inequality, more precisely for the application of the density result stated in Theorem \ref{th: crismale-density2}, \BBB we make the following geometrical assumption on the Dirichlet boundary $\partial_D \Omega= \Omega' \cap \partial \Omega$: there exists a decomposition $\partial \Omega = \partial_D \Omega \cup \partial_N\Omega \cup N$ with 
\begin{align}\label{eq: density-condition2}
\partial_D \Omega, \partial_N\Omega \text{ relatively open}, \ \  \  \mathcal{H}^{d-1}(N) = 0, \ \ \  \partial_D\Omega \cap \partial_N \Omega = \emptyset, \ \ \  \partial (\partial_D \Omega) = \partial (\partial_N \Omega),
\end{align}
and there exist $\bar{\delta}>0$ small and $x_0 \in\R^d$ such that for all $\delta \in (0,\bar{\delta})$ there  holds  
\begin{align}\label{eq: density-condition}
O_{\delta,x_0}(\partial_D \BBB \Omega \EEE ) \subset \Omega,
\end{align}
where $O_{\delta,x_0}(x) := x_0 + (1-\delta)(x-x_0)$. \EEE 

We now present our main $\Gamma$-convergence result.  Recall  Definition \ref{def:conv}, as well as the definition of the nonlinear energies in \eqref{rig-eq: Griffith en} and  \eqref{eq: relaxed energy}. Moreover, recall     the spaces $\mathcal{S}_{\eps,h}$ and $GSBD^2_h$ in \eqref{eq: boundary-spaces} for $h \in W^{2,\infty}(\Omega';\R^d)$.

\begin{theorem}[Passage to linearized model]\label{rig-th: gammaconv}
Let $\Omega \subset \Omega' \subset \R^d$ be bounded Lipschitz domains. Suppose that $W$ satisfies \eqref{assumptions-W} and that  \eqref{eq: density-condition2}-\eqref{eq: density-condition} hold.  Let $h \in W^{2,\infty}(\Omega';\R^d)$.  

\begin{itemize}
\item[(a)] (Compactness) For each sequence $(y_\eps)_\eps$ with $y_\eps \in \mathcal{S}_{\eps,h}$ and $\sup_\eps \mathcal{E}_\eps(y_\eps) < +\infty$, there exists a subsequence (not relabeled) and $u \in GSBD^2_h$ such that $y_\eps \rightsquigarrow u$.
\item[(b)] ($\Gamma$-liminf inequality) For each sequence  $(y_\eps)_\eps$, $y_\eps \in \mathcal{S}_{\eps,h}$, with $y_\eps \rightsquigarrow u$ for some $u \in GSBD^2_h$ we have   
$$\liminf_{\eps \to 0} \mathcal{E}_{\eps}(y_\eps) \ge \mathcal{E}(u).$$
\item[(c)] ($\Gamma$-limsup inequality) For each $u \in GSBD^2_h$ there exists a sequence $(y_\eps)_\eps$, $y_\eps \in \mathcal{S}_{\eps,h}$, such that $y_\eps \rightsquigarrow u$  and  
  $$\lim_{\eps \to 0 } \mathcal{E}_{\eps}(y_\eps) = \mathcal{E}(u).$$
\end{itemize}
The same \BBB statements hold \EEE  with $\overline{\mathcal{E}}_\eps$ in place of  $\mathcal{E}_\eps$.
\end{theorem} 

We point out that we identify  a `standard' Griffith energy in linearized elasticity although we departed from a nonlinear model for nonsimple materials. As a corollary, we obtain the convergence of minimizers for  boundary value problems.

\begin{corollary}[Minimization problems]\label{cor: main cor}
Consider the setting of Theorem \ref{rig-th: gammaconv}.  Then 
\begin{align}\label{eq: eps control2}
\inf_{\bar{y} \in \mathcal{S}_{\eps,h}} \mathcal{E}_\eps(\bar{y}) \  \to \  \min_{u \in GSBD^2_h} \mathcal{E}(u)
\end{align}
as $\eps \to 0$. Moreover, for each sequence $(y_\eps)_\eps$ with $y_\eps \in \mathcal{S}_{\eps,h}$ satisfying
\begin{align}\label{eq: eps control}
\mathcal{E}_\eps(y_\eps) \le \inf_{\bar{y} \in \mathcal{S}_{\eps,h}} \mathcal{E}_\eps(\bar{y}) + \rho_\eps \  \ \ \ \ \text{with } \ \rho_\eps \to 0 \text{ as $\eps \to 0$},   
\end{align}
there exist a subsequence (not relabeled) and $u \in GSBD^2_h$ with $\mathcal{E}(u) = \min_{v \in GSBD^2_h} \mathcal{E}(v)$  such that $y_\eps \rightsquigarrow u$.  
\end{corollary}

The results announced in this subsection will be proved in Subsection \ref{sec: gamma}.

\section{Preliminaries}\label{sec:pre}

In this section we collect some fundamental properties about (generalized) special functions of bounded variation and deformation. In particular, we recall and prove some results for $GSBV^2_2$ and $GSBD^2$ that will be needed for the proofs in Section \ref{sec: proofs}.

 \subsection{Caccioppoli partitions}\label{sec: Caccio}
We say that a partition $(P_j)_j$ of an open set $\Omega\subset \R^d$ is a \textit{Caccioppoli partition} of $\Omega$ if $\sum\nolimits_j \mathcal{H}^{d-1}(\partial^* P_j) < + \infty$, where $\partial^* P_j$ denotes  the \emph{essential boundary} of $P_j$ (see \cite[Definition 3.60]{Ambrosio-Fusco-Pallara:2000}). The  local structure of Caccioppoli partitions can be characterized as follows (see \cite[Theorem 4.17]{Ambrosio-Fusco-Pallara:2000}).
 
\begin{theorem}\label{th: local structure}
Let $(P_j)_j$ be a Caccioppoli partition of $\Omega$. Then 
$$\bigcup\nolimits_j (P_j)^1 \cup \bigcup\nolimits_{i \neq j} (\partial^* P_i \cap \partial^* P_j)$$
contains $\mathcal{H}^{d-1}$-almost all of $\Omega$.
\end{theorem}
Here, $(P)^1$ denote the points where $P$ has density one (see again \cite[Definition 3.60]{Ambrosio-Fusco-Pallara:2000}). 
 Essentially, the  theorem states that $\mathcal{H}^{d-1}$-a.e.\ point of $\Omega$ either belongs to exactly one element of the partition or to the intersection of exactly two sets $\partial^* P_i$, $\partial^* P_j$.

\subsection{$GSBV^2$ and $GSBV^2_2$ functions}
For the general notions on $SBV$ and $GSBV$ functions and their  properties we refer to \cite[Section 4]{Ambrosio-Fusco-Pallara:2000}. For  $\Omega \subset \R^d$ open and $m \in \N$, we define $GSBV^2(\Omega;\R^m)$ as in  \eqref{eq: space2}, for general $m$.  We denote by $\nabla y$ the approximate differential and by  $J_y$  the set of approximate jump points of $y$, which is an $\mathcal{H}^{d-1}$-rectifiable set.  We recall that $GSBV^2(\Omega;\R^m)$ is a vector space, see  \cite[Proposition 2.3]{DalMaso-Francfort-Toader:2005}. \BBB In a similar fashion, we  say $y \in SBV^2(\Omega;\R^m)$ if $y \in SBV(\Omega;\R^m)$, $\nabla y \in L^2(\Omega;\R^{m \times d})$, and $\mathcal{H}^{d-1}(J_y)< + \infty$. \EEE

We define  $GSBV^2_2(\Omega;\R^m)$ as in  \eqref{eq: space}, for general $m$. For $m=1$ we write $GSBV^2_2(\Omega)$.   \BBB By definition, \EEE $\nabla y \in GSBV^2(\Omega;\R^{m \times d})$, and we use the notation  $\nabla^2 y$ and $J_{\nabla y}$ for the approximate differential and the jump set of $\nabla y$, respectively. Applying \cite[Proposition 2.3]{DalMaso-Francfort-Toader:2005}  on $y$ and $\nabla y$, we find that $GSBV^2_2(\Omega;\R^m)$ is a vector space. The following  result is the key ingredient for the proof of  Proposition \ref{prop: relaxation}. 

\begin{theorem}[Compactness in $GSBV^2_2$]\label{th: GSBV2 comp}
Let $\Omega \subset \R^d$ be open and bounded, and let $m \in \N$.  Let $(y_n)_n$ be a sequence in $GSBV^2_2(\Omega;\R^m)$. Suppose that there exists a  continuous, increasing function $\psi: [0,\infty) \to [0,\infty)$ with $\lim_{t \to \infty} \psi(t) = + \infty$ such that 
$$\sup_{n \in \N}  \Big(\int_{\Omega} \psi(|y_n|)\, dx + \int_{\Omega} |\nabla^2 y_n|^2   \, dx  + \mathcal{H}^{d-1}(J_{y_n} \cup J_{\nabla y_n}) \Big) < + \infty.$$
Then there exist a subsequence, still denoted by $(y_n)_n$, and a function $y \in [GSBV(\Omega)]^m$ with $\nabla y \in GSBV^2(\Omega;\R^{m \times d})$ such that for all \BBB $0 < \gamma_2 \le \gamma_1 \le 2\gamma_2$ \EEE there holds
\begin{align}\label{eq: main compi-resu}
{\rm (i)} & \ \ y_n \to y \text{ a.e.\ in $\Omega$}, \notag\\
{\rm (ii)} & \ \ \nabla y_n \to \nabla y \text{ a.e.\ $\Omega$},\notag\\
{\rm (iii)} & \ \  \nabla^2 y_n \rightharpoonup \nabla^2 y \text{ weakly in $L^2(\Omega;\R^{m \times d \times d})$},\notag\\
{\rm (iv)} & \ \ \gamma_1\mathcal{H}^{d-1}(J_y) + \gamma_2 \mathcal{H}^{d-1}(J_{\nabla y} \setminus J_y) \le \liminf_{n \to \infty} \big(   \gamma_1\mathcal{H}^{d-1}(J_{y_n}) + \gamma_2 \mathcal{H}^{d-1}(J_{\nabla y_n} \setminus J_{y_n})\big).
\end{align}
If in addition $\sup_{n \in \N} \Vert \nabla y_n \Vert_{L^2(\Omega)} < + \infty$, then $y \in  GSBV^2_2(\Omega;\R^m)$.
\end{theorem}

\begin{proof}
First, we observe that it suffices to treat the case $m=1$ since otherwise one may argue componentwise,  \BBB see particularly \cite[Lemma 3.1]{Francfort-Larsen:2003} how to deal with property (iv). \EEE The result has been proved in \cite[Theorem 4.4, Theorem 5.13, Remark 5.14]{Carriero2} with the only difference that we \BBB just \EEE assume $\sup_{n \in \N}  \int_{\Omega} \psi(|y_n|)\, dx < + \infty $ here instead of $\sup_{n \in \N} \Vert y_n \Vert_{L^2(\Omega)} < + \infty$. We briefly indicate the necessary adaptions in the proof of \cite[Theorem 4.4]{Carriero2} for $m=1$. \BBB To ease comparison with \cite{Carriero2}, we point out that in that paper the notation $GSBV^2(\Omega)$ is used for functions $u$  \EEE with $u \in GSBV(\Omega)$ and $\nabla u \in [GSBV(\Omega)]^d$.  

For $k \in \N$, we define some  $\varphi_k \in C^2(\R)$ by $\varphi_k(t) = t$ for $t \in [-k+1,k-1]$, $|\varphi_k(t)| = k $ for $|t| > k+1$, and $0 \le \varphi_k' \le 1$. By $\Vert \varphi_k \circ y_n \Vert_{L^1(\Omega)} \le k \mathcal{L}^d(\Omega) $ and by using an interpolation inequality one can check that $(\varphi_k \circ y_n)_n$ is bounded in $BV_{\rm loc}(\Omega)$, see \cite[(4.8)]{Carriero2}. Therefore, \BBB by a diagonal argument  there exist a subsequence of $(y_n)_n$ and \EEE  functions $w_k \in BV_{\rm loc}(\Omega)$ \BBB for all $k \in \N$ \EEE such that
\begin{align}\label{eq: compi1}
\varphi_k \circ y_n \to  w_k \ \ \ \text{a.e.\ in $\Omega$ \, for all $k \in \N$}.
\end{align}
Since $\psi$ is continuous and increasing, and $|\varphi_k(t)| \le |t|$ for all $t \in \R$, we also get by Fatou's lemma
\begin{align}\label{eq: compi2}
\Vert \psi(|w_k|) \Vert_{L^1(\Omega)} \le \liminf_{n \to \infty} \Vert \psi(|\varphi_k \circ y_n|) \Vert_{L^1(\Omega)} \le \sup_{n \in \N} \int_{\Omega} \psi(|y_n|)\, dx < + \infty.
\end{align}
Let $E_k = \lbrace |w_k| < k-1 \rbrace$.  The properties of $\varphi_k$ along with \eqref{eq: compi1} imply
\begin{align}\label{eq: compi3}
y_n \to w_k  \ \ \ \text{a.e.\ in $E_k$ for all $k \in \N$}, \ \ \ \ \     w_k = w_l \ \ \    \text{ on $E_k$ for all $k \le l$. }  
\end{align}
By using \eqref{eq: compi2} we  observe that  $\mathcal{L}^d(\Omega \setminus E_k) \to 0$ as $k \to \infty$ since $\lim_{t \to \infty} \psi(t) = + \infty$. This  together with  \eqref{eq: compi3} shows that the measurable function $y: \Omega \to \R$ defined by $y := \lim_{k\to \infty} w_k$ satisfies $y = w_k$ on $E_k$ for all $k \in \N$ and therefore 
$$y_n \to y  \ \ \ \text{a.e.\ in $\Omega$}.$$    
The rest of the proof starting with \cite[(4.10)]{Carriero2} remains unchanged. In \cite{Carriero2}, it has been shown that $y \in GSBV(\Omega)$ and $\nabla y \in [GSBV(\Omega)]^d$. Since $ \nabla^2 y\in L^2(\Omega;\R^{d\times d})$ and $\mathcal{H}^{d-1}(J_{\nabla y}) < +\infty$, we actually get $\nabla y \in GSBV^2(\Omega;\R^d)$. Finally, given an additional control on $(\nabla y_n)_n$ in $L^2$, we also find $ \nabla y \in L^2(\Omega;\R^d)$  and $\mathcal{H}^{d-1}(J_y) < +\infty$. This implies $y\in GSBV^2_2(\Omega)$,  see \eqref{eq: space}. 
\end{proof}

We now proceed with a version of Theorem \ref{th: GSBV2 comp} without a priori bounds on the functions. We also take boundary data into account.  The result relies on Theorem \ref{th: GSBV2 comp} and \cite{Manuel}.

\begin{theorem}[Compactness in $GSBV^2_2$ without a priori bounds]\label{th: comp}
Let $\Omega \subset \Omega' \subset \R^d$ be bounded Lipschitz domains, and let $m\in\N$. Let $g \in W^{2,\infty}(\Omega';\R^m)$. Consider $(y_n)_n  \subset GSBV^2_2(\Omega';\R^m)$ with $y_n = g$ on $\Omega' \setminus \overline{\Omega}$ and 
$$\sup_{n \in \N}  \Big(\int_{\BBB \Omega' \EEE } \big( |\nabla y_n|^2 + |\nabla^2 y_n|^2 \big) \, dx  + \mathcal{H}^{d-1}(J_{y_n} \cup J_{\nabla y_n}) \Big) < + \infty.$$

\noindent Then we find a  subsequence (not relabeled), modifications $(z_n)_n \subset GSBV^2_2(\Omega';\R^m)$ satisfying $z_n = g$  on $\Omega' \setminus \overline{\Omega}$ and
\begin{align}\label{eq: compi1-new}
{\rm (i)} & \ \ z_n = g \text{ on }   S_n:= \lbrace \nabla z_n \neq  \nabla  y_n \rbrace \cup  \lbrace \nabla^2 z_n \neq \nabla^2  y_n \rbrace, \ \ \ \ \text{where $\mathcal{L}^d(S_n) \to 0$ as $n \to \infty$},  \notag\\
{\rm (ii)} & \ \  \BBB \lim_{n \to \infty} \EEE \mathcal{H}^{d-1}\big(  \big( J_{z_n} \cup J_{\nabla z_n}  \big)\setminus  \big( J_{y_n} \cup J_{\nabla y_n}  \big)  \big) = 0,
\end{align}
as well as a limiting function  $y \in GSBV^2_2(\Omega';\R^m)$ with $y = g$ on $\Omega' \setminus \overline{\Omega}$ such that 
\begin{align}\label{eq: compi2-new}
{\rm (i)}& \ \ \text{$z_n \to y$ in measure on $\Omega'$,}\notag\\ 
{\rm (ii)}& \ \  \nabla z_n \to \nabla y \text{ a.e.\ $\Omega'$ and $\nabla z_n  \rightharpoonup \nabla y$   weakly in $L^2(\Omega'; \R^{m\times d})$}\notag\\
{\rm (iii)}& \ \ \text{$\nabla^2 z_n  \rightharpoonup \nabla^2 y$   weakly in $L^2(\Omega'; \R^{m\times d \times d})$}\notag\\
{\rm (iv)} & \ \ \mathcal{H}^{d-1}(J_y \cup J_{\nabla y}) \le \liminf_{n \to \infty}  \mathcal{H}^{d-1}(J_{z_n} \cup J_{\nabla z_n} ).
\end{align}

\end{theorem}

In general, it is indispensable to pass to modifications. Consider, e.g., the  sequence $y_n = n \chi_U$ for some set $U \subset \Omega$ of finite perimeter. The idea in  \cite[Theorem 3.1]{Manuel}, where this result is proved in the space $GSBV^2(\Omega;\R^m)$,  relies on constructing modifications $(z_n)_n$ by (cf.\ \cite[(37)-(38)]{Manuel})
\begin{align}\label{eq:modifications}
z_n =  g\chi_{R_n} +  \sum\nolimits_{j \ge 1} (y_n - t^{n}_j) \chi_{P_j^{n}}
\end{align}
for Caccioppoli partitions  $\Omega' = \bigcup_{j \ge 1} P_j^{n} \cup R_n$, and suitable translations  $(t_j^{n})_{j \ge 1} \subset \R^m$, where 
\begin{align}\label{eq:modifications2}
{\rm (i)}  & \ \  \lim_{n\to \infty}  \mathcal{L}^d(R_n) =0, \notag \\ 
   {\rm (ii)} & \ \  \lim_{n\to \infty} \mathcal{H}^{d-1}(J_{z_n} \setminus J_{y_n}) =  \lim_{n\to \infty} \mathcal{H}^{d-1}  \big((\partial^* R_n \cap \Omega') \setminus J_{y_n}\big) =0.
\end{align}

\begin{proof}[Proof of Theorem \ref{th: comp}]
We briefly indicate the necessary adaptions with respect to \cite[Theorem 3.1]{Manuel}  to obtain the result in the frame of $GSBV^2_2(\Omega';\R^m)$ involving second derivatives. First, by \cite[Theorem 3.1]{Manuel} we find modifications $(z_n)_n$ as in \eqref{eq:modifications} satisfying $z_n = g$  on $\Omega' \setminus \overline{\Omega}$ and   $y \in GSBV^2(\Omega';\R^m)$ such that $z_n \to y$ in measure on $\Omega'$, up to passing to a subsequence.   By \eqref{eq:modifications2} we get  \eqref{eq: compi1-new}.

As $z_n \to y$ in measure on $\Omega'$,  \cite[Remark 2.2]{Solombrino} implies that there   exists a  continuous, increasing function $\psi: [0,\infty) \to [0,\infty)$ with $\lim_{t \to \infty} \psi(t) = + \infty$ such that up to subsequence (not relabeled)  $\sup_{n \in \N}\int_{\Omega'} \psi(|z_n|)\, dx < + \infty$. \BBB Moreover, by the assumptions on $y_n$, \eqref{eq: compi1-new}, and the fact that $g \in W^{2,\infty}(\Omega';\R^m)$ we get that $\nabla z_n$ and $\nabla^2 z_n$ are uniformly controlled in $L^2$, as well as $\sup_{n \in \N} \mathcal{H}^{d-1}(J_{z_n} \cup J_{\nabla z_n})<+\infty$. \EEE Then Theorem \ref{th: GSBV2 comp} yields  $y \in GSBV^2_2(\Omega';\R^m)$. Along with \eqref{eq: main compi-resu} for $\gamma_1 = \gamma_2 $ we also get \eqref{eq: compi2-new}, apart from the weak convergence of $(\nabla z_n)_n$. The weak convergence readily follows from   $\sup_{n \in \N} \Vert \nabla z_n\Vert_{L^2(\Omega')} \le \sup_{n \in \N} \Vert \nabla y_n\Vert_{L^2(\Omega')} + \BBB \Vert \nabla g\Vert_{L^2(\Omega')} \EEE  < +\infty$. 
\end{proof}

\subsection{$GSBD^2$ functions}\label{sec: GSBD}
 We refer the reader to \cite{ACD} and \cite{DalMaso:13} for the definition, notations,  and basic properties of $SBD$ and $GSBD$ functions, respectively. Here, we only recall briefly  some  relevant notions which can be defined for generalized  functions of bounded deformation: let $\Omega \subset \R^d$ open and bounded. In \cite[Theorem 6.2 and Theorem 9.1]{DalMaso:13} it is shown that  for $ u \in GSBD(\Omega)$ the jump set $J_u$  is $\mathcal{H}^{d-1}$-rectifiable and \BBB that an \EEE approximate symmetric differential $e(u)(x)$ exists at $\mathcal L^{d}$-a.e.\ $x\in \Omega$. We define the space $GSBD^2(\Omega)$ by
$$
GSBD^2(\Omega):= \{u \in GSBD(\Omega): e(u) \in L^2 (\Omega; \mathbb R_{\mathrm{sym}}^{d\times d})\,,\,\mathcal{H}^{d-1}(J_u) < +\infty\}\,.
$$
The space $GSBD^2(\Omega)$ is a vector subspace of the vector space of $\mathcal{L}^d$-measurable function, see \cite[Remark 4.6]{DalMaso:13}. Moreover, there holds $GSBV^2(\Omega;\R^d) \subset GSBD^2(\Omega)$. The following compactness result in $GSBD^2$ has been proved in \cite{Crismale}.

\begin{theorem}[$GSBD^2$ compactness]\label{th: crismale-comp}
\BBB Let $\Omega \subset \R^d$ be open, bounded. \EEE Let $(u_n)_n \subset  GSBD^2(\Omega)$ be a sequence satisfying
$$ \sup\nolimits_{n\in \N} \big( \Vert e(u_n) \Vert_{L^2(\Omega)} + \mathcal{H}^{d-1}(J_{u_n})\big) < + \infty.$$
Then there exists a subsequence (not relabeled) such that the set  $A := \lbrace x\in \Omega: \, |u_n(x)| \to \infty \rbrace$ has finite perimeter, and \BBB there exists \EEE $u \in GSBD^2(\Omega)$ such that 
\begin{align}\label{eq: at crsimale comp} 
{\rm (i)} & \ \ u_n \to u \  \ \ \ \text{ in measure on } \Omega \setminus A, \notag \\ 
{\rm (ii)} & \ \ e(u_n) \rightharpoonup e(u) \ \ \ \text{ weakly in } L^2(\Omega \setminus A; \R^{d \times d}_{\rm sym}),\notag \\
{\rm (iii)} & \ \ \liminf_{n \to \infty} \mathcal{H}^{d-1}(J_{u_n}) \ge \mathcal{H}^{d-1}(J_u \cup \BBB (\partial^*A \cap \Omega) \EEE ).
\end{align}
\end{theorem}

We briefly remark that \RRR \eqref{eq: at crsimale comp}(i) \EEE   is slightly weaker with respect to \eqref{eq: compi2-new}(i) in Theorem \ref{th: comp} (or the corresponding version in $GSBV$, see \cite{Manuel}) in the sense that there might be a set $A$ where the sequence $(u_n)_n$ is \emph{unbounded}, cf.\ the example below Theorem \ref{th: comp}. This phenomenon is avoided in Theorem \ref{th: comp} by passing to suitable modifications which consists in subtracting piecewise constant functions, see \eqref{eq:modifications}. We point out that an analogous result in $GSBD^2$ is so far only available in dimension two, see \cite[Theorem 6.1]{Solombrino}. We now state two density results.

\begin{theorem}[Density]\label{th: crismale-density}
Let $\Omega \subset \R^d$ be a bounded Lipschitz domain.  Let $u \in GSBD^2(\Omega)$. Then there exists a sequence $(u_n)_n \subset  SBV^2(\Omega; \R^d) \cap L^\infty(\Omega;\R^d)$ such that each $J_{u_n}$ is closed   and included in a finite union of closed connected pieces of $C^1$ hypersurfaces, each $u_n$ belongs to
$C^{\infty}(\overline{\Omega} \setminus J_{u_n};\R^d) \cap W^{m,\infty}({\Omega} \setminus J_{u_n};\R^d) $ for every $m \in \N$, and the following properties hold:
\begin{align*}
\begin{split}
{\rm (i)} & \ \ u_n \to  u  \text{ in measure on } \Omega,\\
{\rm (ii)} & \ \ \Vert e(u_n) - e(u) \Vert_{L^2(\Omega)} \to 0,\\
{\rm (iii)} &  \ \ \mathcal{H}^{d-1}(J_{u_n} \triangle J_u) \to  0.
\end{split}
\end{align*}
\end{theorem}

\begin{proof}
The result follows by combining \cite[Theorem 1.1]{Crismale2} and \cite[Theorem 1.1]{Crismale3}. First, \cite[Theorem 1.1]{Crismale2} yields an approximation $u_n$ satisfying   $u_n \in \BBB SBV^2(\Omega;\R^d) \cap \EEE W^{1,\infty}({\Omega} \setminus J_{u_n};\R^d)$, and \BBB then \EEE \cite[Theorem 1.1]{Crismale3} gives the higher regularity. 
\end{proof}

 An adaption of the proof allows to impose boundary conditions on the approximating sequence. Suppose that the Lipschitz domains $\Omega \subset \Omega'$ satisfy the conditions introduced in \eqref{eq: density-condition2}-\eqref{eq: density-condition}. By $\mathcal{W}(\Omega;\R^d)$ we denote the space of all functions
$u \in SBV(\Omega;\R^d)$ such that $J_u$ is a finite union of disjoint $(d-1)$-simplices and $u \in W^{k,\infty}(\Omega \setminus J_u; \R^d)$ for every $k \in \N$.

\begin{theorem}[Density with boundary data]\label{th: crismale-density2}
Let $\Omega \subset \Omega' \subset  \R^d$  be bounded Lipschitz domains satisfying \eqref{eq: density-condition2}-\eqref{eq: density-condition}.  Let $g \in W^{r,\infty}(\Omega')$ for $r \in \N$.    Let $u \in GSBD^2(\Omega')$ with $u = g$ on $\Omega' \setminus \overline{\Omega}$. Then there exists a sequence of functions $(u_n)_n \subset  SBV^2(\Omega; \R^d)$, a sequence of neighborhoods $(U_n)_n \subset \Omega'$ of $\Omega' \setminus \Omega$, and a sequence of neighborhoods $(\Omega_n)_n \subset \Omega$ of $\Omega \setminus U_n$  such that $u_n = g$ on $\Omega' \setminus \overline{\Omega}$, $u_n|_{U_n} \in W^{r,\infty}(U_n;\R^d)$, and $u_n|_{\Omega_n} \in \mathcal{W}(\Omega_n;\R^d)$, and  the following properties hold:
\begin{align}\label{eq: dense-boundary}
{\rm (i)} & \ \ u_n \to  u  \text{ in measure on } \Omega',\notag\\
{\rm(ii)} & \ \ \Vert e(u_n) - e(u) \Vert_{L^2(\Omega')} \to 0,\notag\\
{\rm (iii)} &  \ \  \mathcal{H}^{d-1}(J_{u_n}) \to \mathcal{H}^{d-1}(J_u).
\end{align}
In particular, $u_n \in W^{r,\infty}(\Omega \setminus J_{u_n};\R^d)$.
\end{theorem}

\begin{proof}
The fact that $u$ can be approximated by a sequence $(u_n)_n \subset SBV^2(\Omega';\R^d) \cap L^\infty(\Omega;\R^d)$ satisfying \eqref{eq: dense-boundary} and $u_n = g$ in a neighborhood of $\Omega' \setminus {\Omega}$ has been addressed in \cite[Proof of Theorem 5.4]{Crismale2}. Here, also the necessity of the geometric assumptions \eqref{eq: density-condition2}-\eqref{eq: density-condition} is discussed, see \cite[Remark 5.6]{Crismale2}. The fact that the approximating sequence can be chosen as in the statement then follows by applying on each $u_n$  a construction  very similar to the one of \cite[Proposition 2.5]{Giacomini:2005} along with a diagonal argument. This construction consists in a suitable cut-off construction and the application of the density result \cite{Cortesani-Toader:1999}. We also refer to \cite[Theorem 3.5]{SchmidtFraternaliOrtiz:2009}    for a similar statement. 
\end{proof}

\section{Proofs}\label{sec: proofs}

This section contains the \BBB proofs \EEE of our results.

\subsection{Relaxation and existence of minimizers for the nonlinear model}\label{sec: proofs-nonlinear}

In this subsection we prove Proposition \ref{prop: relaxation} and Theorem \ref{th: existence}.

\begin{proof}[Proof of Proposition \ref{prop: relaxation}]
For $y \in GSBV_2^2(\Omega;\R^d)$ we define   
\begin{align}\label{eq: 5L}
\mathcal{E}'_\eps(y) = \inf \big\{ \liminf\nolimits_{n \to \infty} \mathcal{E}_\eps(y_n,\Omega): y_n \to y \ \text{\rm  in measure on $\Omega$}\big\},
\end{align}
 and define $\overline{\mathcal{E}}_\eps(\cdot,\Omega)$ as in  \eqref{eq: relaxed energy}. We need to check  that $\mathcal{E}'_\eps = \overline{\mathcal{E}}_\eps(\cdot,\Omega)$. In the proof, we write $\tilde{\subset}$ and $\tilde{=}$ for brevity if the inclusion or the identity holds up to an $\mathcal{H}^{d-1}$-negligible set, respectively.
 
 \emph{Step 1: $\mathcal{E}'_\eps \ge \overline{\mathcal{E}}_\eps(\cdot,\Omega)$.} Since by definition $\overline{\mathcal{E}}_\eps(y,\Omega) \le {\mathcal{E}}_\eps(y,\Omega)$ for all $y \in GSBV^2_2(\Omega;\R^d)$, see   \eqref{rig-eq: Griffith en}, it suffices to confirm that $\overline{\mathcal{E}}_\eps(\cdot,\Omega)$ is lower semicontinous with respect to the convergence in measure. To see this, consider $(y_n)_n \subset GSBV^2_2(\Omega;\R^d)$ with $y_n \to y$ in measure $\Omega$ and $\sup_{n\in \N}\overline{\mathcal{E}}_\eps(y_n,\Omega) <+ \infty$. By using \cite[Remark 2.2]{Solombrino}, there   exists a  continuous, increasing function $\psi: [0,\infty) \to [0,\infty)$ with $\lim_{t \to \infty} \psi(t) = + \infty$ such that up to subsequence (not relabeled)  $\sup_{n \in \N}\int_{\Omega} \psi(|y_n|)\, dx < + \infty$. Then from Theorem \ref{th: GSBV2 comp} we obtain 
$$\overline{\mathcal{E}}_\eps(y,\Omega) \le \liminf_{n \to \infty} \overline{\mathcal{E}}_\eps(y_n,\Omega). $$
In fact, for the second and the third term in \eqref{eq: relaxed energy} we use \eqref{eq: main compi-resu}(iii) and (iv) \BBB for $\gamma_1=\gamma_2$, \EEE respectively. The first term in \eqref{eq: relaxed energy}  is lower semicontinuous by the continuity of $W$, \eqref{eq: main compi-resu}(ii), and Fatou's lemma. This shows that $\overline{\mathcal{E}}_\eps(\cdot,\Omega)$ is lower semicontinous and concludes the proof of  $\mathcal{E}'_\eps \ge \overline{\mathcal{E}}_\eps(\cdot,\Omega)$.

 \emph{Step 2: $\mathcal{E}'_\eps \le \overline{\mathcal{E}}_\eps(\cdot,\Omega)$.} In the proof, we will use the following argument several times: if $y_1,y_2 \in GSBV^2(\Omega;\R^d)$, then for a.e.\ $t \in \R$ there holds that $z:= y_1 + ty_2\in GSBV^2(\Omega;\R^d)$ satisfies $J_z = J_{y_1} \cup J_{y_2}$, see \cite[Proof of Lemma 3.1]{Francfort-Larsen:2003} or \cite[Proof of Lemma 4.5]{DalMaso-Francfort-Toader:2005} for such an argument. We point out that here  we exploit the fact that $GSBV^2(\Omega;\R^d)$ is a vector space.

 Observe that for  each $y \in GSBV^2_2(\Omega;\R^d)$ and each $\nu \in S^{d-1}$, the function $v := \nabla y \cdot \nu$   lies in $GSBV^2(\Omega;\R^d) \subset GSBD^2(\Omega)$. We can choose $\nu \in S^{d-1}$ such that  there holds $\mathcal{H}^{d-1}(J_v \triangle J_{\nabla y}) = 0$. We apply Theorem \ref{th: crismale-density} to approximate $v \in GSBD^2(\Omega)$ by a sequence   $(v_n)_n \subset  SBV^2(\Omega; \R^d)$ such that  $v_n \in W^{2,\infty}({\Omega} \setminus J_{v_n};\R^d)$ and 
\begin{align}\label{eq: jumpi1} 
\mathcal{H}^{d-1}(J_{v_n} \triangle J_{\nabla y} ) = \mathcal{H}^{d-1}(J_{v_n} \triangle J_{v} ) \to 0
\end{align}
as $n \to \infty$. We point out that $J_{\nabla v_n} \tilde{\subset} J_{v_n}$ since   $v_n \in W^{2,\infty}({\Omega} \setminus J_{v_n};\R^d)$.  \BBB Using \EEE  $v_n \in W^{2,\infty}({\Omega} \setminus J_{v_n};\R^d)$ we can choose  a  sequence $(\eta_n)_n$ with $\eta_n \to 0$ such that $z_n := y + \eta_n v_n \in GSBV^2_2(\Omega;\R^d)$ satisfies $J_{z_n} \tilde{=} J_y \cup J_{v_n}$ \BBB and there holds \EEE $z_n \to y$ in measure on $\Omega$. \BBB By \eqref{eq: jumpi1}, the continuity of $W$, $J_{z_n} \tilde{=} J_y \cup J_{v_n}$, and $J_{\nabla z_n} \tilde{\subset} J_{\nabla y} \cup J_{v_n}$ we get \EEE
\begin{align}\label{eq: jumpi4}
\limsup\nolimits_{n\to \infty}\overline{\mathcal{E}}_\eps(z_n,\Omega) \le \overline{\mathcal{E}}_\eps(y,\Omega). 
\end{align}
 As $J_{z_n} \tilde{=} J_y \cup J_{v_n}$, $J_{\nabla y} \tilde{=} J_{v}$, and $J_{\nabla v_n} \tilde{\subset} J_{v_n}$, we also get   
\begin{align}\label{eq: jumpi3}
J_{\nabla z_n} \setminus J_{z_n} \,\tilde{\subset}\,  (J_{\nabla y} \cup J_{\nabla v_n}) \setminus  (J_y \cup J_{v_n}) \, \tilde{\subset} \, J_v \setminus J_{v_n}.
\end{align}
In view of \eqref{eq: jumpi1}, \BBB by a Besicovitch covering argument \EEE we can cover the rectifiable sets $J_v \setminus J_{v_n}$ by sets of finite perimeter  \BBB $(E_n)_n \subset \subset \Omega$, \EEE each of which being a countable union of balls with radii smaller than $\frac{1}{n}$, such that 
\begin{align}\label{eq: jumpi2}
\mathcal{L}^d(E_n) + \mathcal{H}^{d-1}(\partial^*  E_n)   \to 0.
\end{align}
We finally define the sequence $y_n \in GSBV^2_2(\Omega;\R^d)$ by $y_n = z_n \chi_{\Omega \setminus E_n} + (\id + b_n) \chi_{E_n} $ for suitable constants $(b_n)_n \subset \R^d$ which are chosen such that $J_{y_n} \tilde{=} (J_{z_n} \setminus E_n) \cup \partial^* E_n$. Now in view of \eqref{eq: jumpi3} and $J_v \setminus J_{v_n} \tilde{\subset} E_n$, we get $J_{\nabla y_n}\tilde{\subset} J_{y_n}$. By \eqref{eq: jumpi2} and $z_n \to y$ in measure on $\Omega$ we get  $y_n \to y$ in measure on $\Omega$.  By  \eqref{assumptions-W}(iii) we obtain  $W(\nabla {y}_n) = 0$,  $ \nabla^2 {y}_n = 0$ on $E_n$. Then by \eqref{eq: relaxed energy},   \eqref{eq: jumpi4}, \eqref{eq: jumpi2}, and the fact that $J_{\nabla y_n}\tilde{\subset} J_{y_n} \tilde{=} (J_{z_n} \setminus E_n) \cup \partial^* E_n$ we get
$$\limsup_{n\to \infty} \overline{\mathcal{E}}_\eps(y_n,\Omega) \le \limsup_{n\to\infty} \big(\overline{\mathcal{E}}_\eps(z_n,\Omega) + \kappa\mathcal{H}^{d-1}(\partial^* E_n) \big) \le \overline{\mathcal{E}}_\eps(y,\Omega).$$
\RRR Since $\overline{\mathcal{E}}_\eps(y_n,\Omega) = \mathcal{E}_\eps(y_n,\Omega)$ for all $n \in \N$ by $J_{\nabla y_n}\tilde{\subset} J_{y_n}$, \eqref{eq: 5L} implies  $\mathcal{E}'_\eps(y) \le \overline{\mathcal{E}}_\eps(y,\Omega)$. This concludes the proof. \EEE
  \end{proof}

\begin{proof}[Proof of Theorem \ref{th: existence}]
We prove the existence of minimizers via the direct method. Let $(y_n)_n \subset GSBV^2_2(\Omega';\R^d)$ with $y_n = g$ on $\Omega' \setminus\overline{\Omega}$ be a minimizing sequence for the minimization problem \eqref{eq: minimization problem}. By \eqref{assumptions-W} we find  $W(F)\ge c_1|F|^2 - c_2$ for $c_1,c_2>0$. Thus, $\sup_{n\in\N} \overline{\mathcal{E}}_\eps(y_n) <+\infty$ also implies $\sup_{n \in \N} \Vert \nabla y_n \Vert_{L^2(\Omega')}< + \infty$, and we can apply Theorem \ref{th: comp}. We obtain a sequence  $(z_n)_n \subset GSBV^2_2(\Omega';\R^d)$ satisfying $z_n = g$ on $\Omega' \setminus \overline{\Omega}$ and a limiting function $y \in GSBV^2_2(\Omega';\R^d)$ with $y = g$ on $\Omega' \setminus \overline{\Omega}$   \BBB such that $z_n \to y$ in measure on $\Omega'$. \EEE Using  \BBB  \eqref{eq: relaxed energy}, \EEE \eqref{eq: compi1-new}, and $g \in W^{2,\infty}(\Omega';\R^d)$  we calculate
\begin{align*}
\limsup_{n \to \infty} \big(\overline{\mathcal{E}}_\eps(z_n)  -\overline{\mathcal{E}}_\eps(y_n) \big)&  \le \limsup_{n \to \infty} \Big( \BBB \eps^{-2} C_{W,g}  \mathcal{L}^d(S_n) + \eps^{-2\beta} \Vert \nabla^2 g\Vert^2_{L^2(S_n)}  \EEE  \\ & \ \ \ \ \ \ \  \ \ \  \ \  \ \ \ \  +  \kappa \big( \mathcal{H}^{d-1}(J_{z_n} \cup J_{\nabla z_n} )- \mathcal{H}^{d-1}(J_{y_n} \cup J_{\nabla y_n} ) \big)  \Big) \RRR\le \EEE 0,
\end{align*}
where \BBB the constant $C_{W,g}$ depends on $W$ and $\Vert \nabla g \Vert_{L^\infty(\Omega')}$. \EEE I.e., $(z_n)_n$ is also a minimizing sequence. By \BBB  $z_n \to y$ in measure on $\Omega'$ and \EEE the fact that  $\overline{\mathcal{E}}_\eps$ is lower semicontinuous with respect to the convergence in measure on $\Omega'$, see  Proposition \ref{prop: relaxation},  we get
$$ \overline{\mathcal{E}}_\eps(y) \le \liminf_{n \to \infty} \overline{\mathcal{E}}_\eps(z_n) \le \liminf_{n \to \infty} \overline{\mathcal{E}}_\eps(y_n) = \inf_{{\bar{y} \in GSBV^2_2(\Omega';\R^d)}} \Big\{ \overline{\mathcal{E}}_\eps(\bar{y}): \  \bar{y} = g \text{ on } \Omega' \setminus \overline{\Omega} \Big\}.
  $$
This shows that $y$ is a minimizer. 
\end{proof}

\subsection{Compactness}\label{sec: comp-proof} 

This subsection is devoted to the proof of     Theorem \ref{thm: compactess}.

\begin{proof}[Proof of Theorem \ref{thm: compactess}(a)]
Consider a sequence $(y_\eps)_\eps$ with $y_\eps \in \mathcal{S}_{\eps,h}$, i.e.,   $y_\eps = \id + \eps h$ on $\Omega' \setminus \overline{\Omega}$. Suppose that $M:=\sup\nolimits_\eps \overline{\mathcal{E}}_\eps(y_\eps)   <+\infty$. We first construct Caccioppoli partitions (Step 1) and the corresponding rotations (Step 2) in order to define $y^{\rm rot}_\eps$. Then we confirm \eqref{eq: first conditions} (Step 3).

\emph{Step 1: Definition of the Caccioppoli partitions.} First, we apply the $BV$ coarea formula (see \cite[Theorem 3.40 or Theorem 4.34]{Ambrosio-Fusco-Pallara:2000}) on each component $(\nabla y_\eps)_{ij} \in GSBV^2(\Omega')$, $1 \le i,j\le d$,  to write 
\begin{align*}
 \int_{-\infty}^\infty \mathcal{H}^{d-1}\big( (\Omega' \setminus  J_{\nabla y_\eps}) \cap \partial^* \lbrace (\nabla y_\eps)_{ij} > t \rbrace \big) \,dt &= |D  (\nabla y_\eps)_{ij}|(\Omega' \setminus  J_{ \nabla y_\eps}) \le \Vert \nabla^2 y_\eps \Vert_{L^1(\Omega')}.
\end{align*}
Using H\"older's inequality and \eqref{eq: relaxed energy} along with $\overline{\mathcal{E}}_\eps(y_\eps) \le M$, we then get
\begin{align}\label{eq: L7XXX}
 \int_{-\infty}^\infty \mathcal{H}^{d-1}\big( (\Omega' \setminus  J_{\nabla y_\eps}) \cap \partial^* \lbrace (\nabla y_\eps)_{ij} > t \rbrace \big) \,dt   \le (\mathcal{L}^d(\Omega'))^{1/2}  \Vert \nabla^2 y_\eps \Vert_{L^2(\Omega')} \le (\mathcal{L}^d(\Omega') M)^{1/2} \eps^\beta. 
\end{align}
Fix $\gamma \in (\frac{2}{3},\beta)$ and define $T_\eps= \eps^{\gamma}$. For all  $k \in \Z$ we find $t^{ij}_k \in (kT_\eps, (k+1)T_\eps]$    such that 
\begin{align}\label{eq: coarea}
\mathcal{H}^{d-1}\big( (\Omega' \setminus J_{\nabla y_\eps}) \cap   \partial^*\lbrace (\nabla y_\eps)_{ij} >t^{ij}_k \rbrace \big) \le \frac{1}{T_\eps} \int_{kT_\eps}^{{(k+1)T_\eps}} \mathcal{H}^{d-1}\big( (\Omega' \setminus J_{\nabla y_\eps}) \cap  \partial^*\lbrace (\nabla y_\eps)_{ij} >t \rbrace \big)\, dt.
\end{align}
Let $G_k^{\eps,ij} = \lbrace (\nabla y_\eps)_{ij} > t^{ij}_{k} \rbrace \setminus \lbrace (\nabla y_\eps)_{ij} > t^{ij}_{k+1} \rbrace$ and note that each set has finite perimeter in $\Omega'$ since it is the difference of two sets of finite perimeter.  Now \RRR \eqref{eq: L7XXX} and  \eqref{eq: coarea} imply \EEE
\begin{align}\label{eq: nummer geben}
\sum\nolimits_{k\in\Z} \mathcal{H}^{d-1}\big( (\Omega' \setminus J_{\nabla y_\eps}) \cap \partial^* G_k^{\eps, ij}  \big) \le 2T^{-1}_\eps  (\mathcal{L}^d(\Omega') M)^{1/2} \eps^\beta \le C \eps^{\beta-\gamma}
\end{align}
for a sufficiently large constant $C>0$ independent of $\eps$. Since $\mathcal{L}^d(\Omega' \setminus \bigcup_{k \in \Z} G_k^{\eps, ij})=0$, $(G_k^{\eps,ij})_{\BBB k\in \Z \EEE }$ is a Caccioppoli partition of $\Omega'$. 
We let $(P^\eps_j)_{j\in\N}$ be the Caccioppoli partition of $\Omega'$ consisting of the nonempty sets of 
$$\big\{ G_{k_{11}}^{\eps, 11} \cap G_{k_{12}}^{\eps, 12}  \cap \ldots \cap  G_{k_{dd}}^{\eps, dd}: \  k_{ij} \in \Z\text{ for } i,j=1,\ldots,d\big\}.$$
Then \BBB  \eqref{eq: nummer geben}  implies \EEE
\begin{align}\label{eq: coarea3}
\sum\nolimits_{j =1}^\infty\mathcal{H}^{d-1}\big(\partial^*P_j^\eps \cap (\Omega' \setminus J_{\nabla y_\eps})   \big) \le C\eps^{\beta-\gamma}
\end{align}
for a constant $C>0$ independent of $\eps$.

\emph{Step 2: Definition of the rotations.}  We now define corresponding rotations. Recalling $T_\eps = \eps^{\gamma}$ we get $|t_k^{ij} -t_{k+1}^{ij}| \le 2T_\eps=2\eps^\gamma$ for all $k \in \Z$, $i,j=1,\ldots,n$. Then by the definition of $G_k^{\eps,ij}$, for each component $P^\eps_j$ of the Caccioppoli partition, we find a matrix $F_j^\eps \in \R^{d\times d}$ such that
\begin{align}\label{eq: coarea4}
\Vert  \nabla y_\eps - F_j^\eps \Vert_{L^\infty(P^\eps_j)} \le c\eps^{\gamma}, 
\end{align}
where $c$ depends only on $d$. \RRR For each $j \in \N$ with $P^\eps_j \subset \Omega$ up to an $\mathcal{L}^d$-negligible set, \EEE we denote by $\bar{R}_j^\eps$ the nearest point projection of $F_j^\eps$ onto $SO(d)$. For all other components $P^\eps_j$, i.e., the components intersecting $\Omega' \setminus \overline{\Omega}$, we set $\bar{R}_j^\eps = \Id$. We now show that for all $j \in \N$ and  for $\mathcal{L}^d$-a.e.\ $x \in P_j^\eps$  there holds
\begin{align}\label{eq: main control}
|\nabla y_\eps(x) - \bar{R}_j^\eps| \le \max\big\{       C\eps^{\gamma}, \  2\dist(\nabla y_\eps(x), SO(d)) \big\}
\end{align}
for a constant $C>0$ independent of $\eps$.

First, we consider components $P^\eps_j$ which are contained in $\Omega$ up to an $\mathcal{L}^d$-negligible set. Recall that  $\bar{R}_j^\eps$ is defined as the nearest point projection of $F_j^\eps$ onto $SO(d)$. If $|\bar{R}_j^\eps- F_j^\eps| \le 3c\eps^{\gamma}$, where $c$ is the constant of  \eqref{eq: coarea4},    \eqref{eq: main control} follows from \eqref{eq: coarea4} and the triangle inequality. Otherwise,  by \eqref{eq: coarea4} we get for $\mathcal{L}^d$-a.e.\ $x \in P_j^\eps$ 
\begin{align*}
\dist(\nabla y_\eps(x), SO(d)) & \ge \dist(F^\eps_j, SO(d)) - c\eps^{\gamma} = |\bar{R}_j^\eps- F_j^\eps|  -  c\eps^{\gamma} \\ &
\ge \tfrac{1}{2} \big(|\bar{R}_j^\eps- F_j^\eps|  + c\eps^{\gamma} \big) \ge \tfrac{1}{2}|\bar{R}_j^\eps- \nabla y_\eps(x)|. 
\end{align*} 
This implies \eqref{eq: main control}. Now consider a component $P^\eps_j$ which intersects $\Omega' \setminus \overline{\Omega}$. Then  by \eqref{eq: coarea4} and the fact that $y_\eps = \id + \eps h$ on $\Omega' \setminus \overline{\Omega}$ there holds 
\begin{align*}
 \Vert  \Id + \eps \nabla h - F_j^\eps \Vert_{L^\infty(P^\eps_j  \setminus \Omega)} \le \Vert  \nabla y_\eps - F_j^\eps \Vert_{L^\infty(P^\eps_j)} \le c\eps^{\gamma}.
\end{align*}
Since $0 < \gamma < 1$, this yields $|F_j^\eps - \Id| \le C\eps^{\gamma}$ for a constant $C$ depending also on $\Vert \nabla h \Vert_{L^\infty(\Omega')}$. This along with \eqref{eq: coarea4} implies \eqref{eq: main control} (for $\bar{R}_j^\eps = \Id$).  We define the rotations in the statement by $R_j^\eps := (\bar{R}^\eps_j)^{-1}$.

\emph{Step 3: Proof of \eqref{eq: first conditions}.} 
We are now in a position to prove \BBB \eqref{eq: first conditions}. \EEE We define $y^{\rm rot}_{\eps}  $ as in \eqref{eq: piecewise rotated}, i.e., $y^{\rm rot}_\eps =      \sum\nolimits_{j=1}^\infty R_j^\eps y_\eps \chi_{P^\eps_j}$. Then \eqref{eq: first conditions}(i) follows from the fact that $y_\eps = \id + \eps h$ on $\Omega' \setminus \overline{\Omega}$ and $y_{\eps}^{\rm rot} = y_\eps$  on $\Omega' \setminus \overline{\Omega}$, where the latter \BBB holds due to $R^\eps_j = \Id$ for all $P_j^\eps$ intersecting $\Omega' \setminus \overline{\Omega}$. \EEE Property \eqref{eq: first conditions}(ii) is a direct consequence of the definition of $y_\eps^{\rm rot}$ and  \eqref{eq: coarea3}. To see \eqref{eq: first conditions}(iv), we use \eqref{eq: main control} and  $R_j^\eps = (\bar{R}^\eps_j)^{-1}$ to get
\begin{align*}
\Vert \nabla y^{\rm rot}_\eps - \Id\Vert^2_{L^2(\Omega')} &= \sum\nolimits_{j=1}^\infty \Vert \nabla y_\eps - \bar{R}^\eps_j\Vert^2_{L^2(P^\eps_j)} \le C\eps^{2\gamma} \mathcal{L}^d(\Omega') +  4 \Vert \dist(\nabla y_\eps, SO(d)) \Vert^2_{L^2(\Omega')} \\&
 \le C(\eps^{2\gamma} + \eps^2) 
\end{align*}
for a constant depending on $M$, where the last step follows from \eqref{assumptions-W}(iii), \eqref{eq: relaxed energy}, and $\overline{\mathcal{E}}_\eps(y_\eps) \le M$.
Since $0 < \gamma < 1$, \eqref{eq: first conditions}(iv) \BBB is proved. \EEE  It remains to show \eqref{eq: first conditions}(iii). We recall the linearization formula  (see \cite[(3.20)]{FrieseckeJamesMueller:02}) 
\begin{align}\label{rig-eq: linearization}
|{\rm sym}(F -\Id)| =  \dist(F,SO(d)) + {\rm O} (|F- \Id|^2)
\end{align}
for $F \in \R^{d \times d}$. By Young's inequality and $|{\rm sym}(F -\Id)|  \le |F -\Id| $ this implies 
\begin{align*}
|{\rm sym}(F -\Id)|^2& \le \min\big\{  |F- \Id|^2, \  C \dist^2(F,SO(d))+ C|F-\Id|^4 \big\}\\ 
& \le   C \dist^2(F,SO(d)) + C\min\big\{ |F- \Id|^2, \ |F-\Id|^4 \big\}.
  \end{align*}
Then we calculate
\begin{align*}
\int_{\Omega'} |{\rm sym}( \nabla y^{\rm rot}_\eps - \Id)|^2 &\le C\int_{\Omega'}  \Big( \dist^2(\nabla y_\eps^{\rm rot},SO(d)) + \min\big\{ |\nabla y_\eps^{\rm rot}- \Id|^2, \ |\nabla y_\eps^{\rm rot}-\Id|^4 \big\}\Big) \\
&\le C\sum\nolimits_{j=1}^\infty\int_{P^\eps_j} \Big(\dist^2(\nabla y_\eps,SO(d)) +  |\nabla y_\eps- \bar{R}^\eps_j|^2 \, \min\big\{ 1, \, |\nabla y_\eps-\bar{R}_j^\eps|^2 \big\}\Big). 
\end{align*}
 By \eqref{eq: main control} we note that for a.e.\ $x \in P_j^\eps$ there holds
$$|\nabla y_\eps(x)- \bar{R}^\eps_j|^2 \,\min\big\{   1 , \, |\nabla y_\eps(x)-\bar{R}_j^\eps|^2 \big\} \le C\eps^{4\gamma} + C\dist^2(\nabla y_\eps(x), SO(d)). $$ 
Here, we used that, if $|\nabla y_\eps(x)-\bar{R}_j^\eps|^2>1$, the maximum in \eqref{eq: main control} is attained for $\dist(\nabla y_\eps(x), SO(d))$, provided that $\eps$ is small enough. Therefore, we get
 \begin{align*}
\int_{\Omega'} |{\rm sym}( \nabla y^{\rm rot}_\eps - \Id)|^2 &\le  C\int_{ \BBB \Omega' \EEE } \dist^2(\nabla y_\eps,SO(d)) + C\mathcal{L}^d(\Omega') \eps^{4\gamma}  \le C\eps^2 + C\eps^{4\gamma},
\end{align*}
where in the last step we have again \BBB used \eqref{assumptions-W}(iii), \eqref{eq: relaxed energy}, \EEE and $\overline{\mathcal{E}}_\eps(y_\eps) \le M$. Since $\gamma >\frac{2}{3} \ge \frac{1}{2}$, we obtain \eqref{eq: first conditions}(iii). This concludes the proof of   Theorem \ref{thm: compactess}(a). 
\end{proof}

 \begin{rem}\label{re: unchanged}
 For later purposes, we point out that the construction shows $y_\eps^{\rm rot} = y_{\eps}$ on all $P_j^\eps$ intersecting $\Omega'\setminus \overline{\Omega}$. 
 \end{rem}
 
\begin{proof}[Proof of Theorem \ref{thm: compactess}(b)]
We define the rescaled displacment fields $u_\eps := \frac{1}{\eps}(y^{\rm rot}_\eps - \id)$ as in \eqref{eq: rescalidipl}.   Clearly, there holds  $u_\eps \in GSBV^2(\Omega';\R^d) \subset GSBD^2(\Omega')$. Note that by \eqref{eq: first conditions}(iii) we obtain $\sup_\eps \Vert e(u_\eps) \Vert_{L^2(\Omega')}  < + \infty$, where for shorthand we again write $e(u_\eps) = \frac{1}{2}(\nabla u_\eps^T + \nabla u_\eps)$. Moreover, in view of  \eqref{eq: first conditions}(ii) and $\beta >\gamma$, we get  
\begin{align}\label{eq: should be sep}
\limsup\nolimits_{\eps \to 0}  \mathcal{H}^{d-1}(J_{u_\eps}) \le \limsup\nolimits_{\eps \to 0}  \mathcal{H}^{d-1}(J_{y_\eps} \cup J_{\nabla y_\eps}) < +\infty.
\end{align}
Therefore, we can apply Theorem \ref{th: crismale-comp} on the sequence $(u_\eps)_\eps$ to obtain $A$ and $u' \in GSBD^2(\Omega')$  such that \eqref{eq: at crsimale comp} holds (up to passing to a subsequence). We first observe that $E_u = A$, where \BBB $E_u := \lbrace x\in \Omega: \, |u_\eps(x)| \to \infty \rbrace$ and $A := \lbrace x\in \Omega': \, |u_\eps(x)| \to \infty \rbrace$.  \EEE To see this, we have to check that  $A \subset \Omega$. This follows from the fact that $u_\eps = h$ on $\Omega' \setminus \overline{\Omega}$ for all $\eps$, see \eqref{eq: first conditions}(i) and \eqref{eq: rescalidipl}.

We define $u := u' \chi_{\Omega' \setminus E_u} + \lambda \chi_{E_u}$ for some $\lambda \in \R^d$ such that $\partial^* E_u \cap \Omega' \subset J_u$ up to an $\mathcal{H}^{d-1}$-negligible set. \RRR Since $J_u \subset J_{u'} \cup (\partial^*E_u\cap \Omega')$, \EEE \eqref{eq: at crsimale comp} then implies  \eqref{eq: the main convergence}, where the last inequality in \eqref{eq: the main convergence}(iii) follows from \eqref{eq: should be sep}. Finally, $u \in GSBD^2_h$  follows from    $u_\eps = h$ on $\Omega' \setminus \overline{\Omega}$  and \eqref{eq: the main convergence}(i).
\end{proof}

\subsection{Passage to linearized model by $\Gamma$-convergence}\label{sec: gamma}

We now give the proof of Theorem \ref{rig-th: gammaconv}.

\begin{proof}[Proof of Theorem \ref{rig-th: gammaconv}] Since $\overline{\mathcal{E}}_\eps \le \mathcal{E}_\eps$, see \eqref{rig-eq: Griffith en} and \eqref{eq: relaxed energy}, the compactness result follows immediately from Theorem \ref{thm: compactess}. It suffices to show the $\Gamma$-liminf inequality for $\overline{\mathcal{E}}_\eps$ and the $\Gamma$-limsup inequality for $\mathcal{E}_\eps$. 

\emph{Step 1: $\Gamma$-liminf inequality.} Consider $u \in GSBD^2_h$ and $(y_\eps)_\eps$, $y_\eps \in \mathcal{S}_{\eps,h}$, such that  $y_\eps \rightsquigarrow u$, i.e, by Definition \ref{def:conv} there exist $y^{\rm rot}_\eps =      \sum\nolimits_{j=1}^\infty R^\eps_j \,  y_\eps \,  \chi_{P^\eps_j}$  and $  u_\eps:  = \frac{1}{\eps} (y^{\rm rot}_\eps - \id) $  such that \eqref{eq: first conditions} and \eqref{eq: the main convergence} hold for some fixed $\gamma \in (\frac{2}{3},\beta)$. The essential step is to prove
\begin{align}\label{eq: essential step}
\liminf_{\eps \to 0} \frac{1}{\eps^2}\int_{\Omega'} W(\nabla y_\eps) \ge \int_{\Omega'} \frac{1}{2} Q(e( u) ).
\end{align}
Once \eqref{eq: essential step} is shown, we conclude by \eqref{eq: relaxed energy} and  \eqref{eq: the main convergence}(iii) that
\begin{align*}
\liminf_{\eps \to 0} \overline{\mathcal{E}}_\eps (y_\eps) \ge  \liminf_{\eps \to 0} \Big(   \frac{1}{\eps^2}\int_{\Omega'} W(\nabla y_\eps) + \kappa\mathcal{H}^{d-1}(J_{y_\eps} \cup J_{\nabla y_\eps}) \Big) \ge  \int_{\Omega'} \frac{1}{2} Q(e(u))  + \kappa\mathcal{H}^{d-1}(J_u).
\end{align*}
In view of \eqref{rig-eq: Griffith en-lim}, this shows $\liminf_{\eps \to 0} \overline{\mathcal{E}}_\eps (y_\eps)  \ge \mathcal{E}(u)$. To see \eqref{eq: essential step}, we first note that the frame indifference of $W$ (see \eqref{assumptions-W}(ii)) and the definitions of $y^{\rm rot}_\eps$ and $u_\eps$ (see  \eqref{eq: piecewise rotated} and \eqref{eq: rescalidipl}) imply 
\begin{align}\label{eq: respresentation}
W(\nabla y_\eps) = W(\nabla y_\eps^{\rm rot}) = W(\Id + \eps \nabla u_\eps). 
\end{align}
In view of $\gamma >2/3$, we can choose  $\eta_\eps \to +\infty$ such that
\begin{align}\label{eq: eta chocie}
\eps^{1-\gamma} \eta_\eps \to  +\infty \ \ \ \text{and} \ \ \  \eps \eta^3_\eps \to 0.
\end{align}
 We define $\chi_\eps \in L^\infty(\Omega')$ by $\chi_\eps(x) = \chi_{[0,\eta_\eps)}(|\nabla u_\eps(x)|)$. \BBB Note that \EEE $\mathcal{L}^d(\lbrace |\nabla u_\eps(x)|> \eta_\eps \rbrace ) \le C(\eps^{\gamma-1}/\eta_\eps)^2$ by \eqref{eq: first conditions}(iv) and the fact that $  u_\eps  = \frac{1}{\eps} (y^{\rm rot}_\eps - \id)$. \BBB Thus, \EEE \eqref{eq: eta chocie} implies       $\chi_\eps \to 1$ boundedly in measure on $\Omega'$. The regularity of $W$ implies  $W(\Id + F) = \frac{1}{2}Q( {\rm sym}(F)) + \omega(F)$, where   $Q$ is defined in \eqref{rig-eq: Griffith en-lim}   and $\omega:\R^{d \times d}\to \BBB \R \EEE $ is a function  satisfying $\BBB |\omega(F)| \EEE \le C|F|^3$  for all $F \in \R^{d\times d}$ with $|F| \le 1$.  Then by \eqref{eq: respresentation} and $W \ge 0$ we get
\begin{align*}
\liminf_{\eps \to 0} \frac{1}{\eps^2}\int_{\Omega'} W(\nabla y_\eps) & \ge \liminf_{\eps\to 0}  \frac{1}{\eps^2}\int_{\Omega'} \chi_\eps W(\Id + \eps \nabla u_\eps)
\\
&=  \liminf_{\eps\to 0} \int_{\Omega'} \chi_\eps \Big( \frac{1}{2}Q(e(u_\eps)) + \frac{1}{\eps^2} \omega(\eps \nabla u_\eps)  \Big) 
\\
&\ge  \liminf_{\eps\to 0} \Big(\int_{\Omega'\setminus E_u} \chi_\eps   \frac{1}{2}Q(e(u_\eps)) + \int_{\Omega'}   \chi_\eps |\nabla u_\eps|^3 \eps \frac{\omega(\eps \nabla u_\eps)}{|\eps \nabla u_\eps|^3}  \Big),
\end{align*}
where $E_u = \lbrace x\in \Omega: \, |u_\eps(x)| \to \infty \rbrace$. 
The second term converges to zero. Indeed, $\chi_\eps\frac{|\omega(\eps \nabla u_\eps)|}{|\eps \nabla u_\eps|^3}$ is uniformly controlled by $C$ and  $\chi_\eps |\nabla u_\eps|^3 \eps$ is uniformly controlled by $\eta_\eps^3 \eps $, where  $\eta_\eps^3 \eps \to 0$ by \eqref{eq: eta chocie}.   As $e( u_\eps) \rightharpoonup e(u)$ weakly in $L^2(\Omega'\setminus E_u,\R^{d\times d}_{\rm sym})$ by \eqref{eq: the main convergence}(ii), $Q$ is convex, and $\chi_\eps$ converges to $1$ boundedly in measure on $\Omega' \setminus E_u$, we conclude 
\begin{align*}
\liminf_{\eps \to 0} \frac{1}{\eps^2}\int_{\Omega'} W(\nabla y_\eps) \ge  \int_{\Omega' \setminus E_u} \frac{1}{2}Q(e(u)) =  \int_{\Omega'} \frac{1}{2}Q(e(u)),
\end{align*}
where the last step follows from the fact that $e(u) = 0$ on $E_u$, see \eqref{eq: the main convergence}(iv). This shows \eqref{eq: essential step} and concludes the proof of the $\Gamma$-liminf inequality. 

\emph{Step 2: $\Gamma$-limsup inequality.}  Consider $u \in GSBD^2_h$ \BBB with $h \in W^{2,\infty}(\Omega';\R^d)$. \EEE Let $\gamma \in (\frac{2}{3},\beta)$.    By Theorem \ref{th: crismale-density2} we  can find a sequence $(v_\eps)_\eps \in GSBV^2_2(\Omega';\R^d)$ with $v_\eps = h$ on $\Omega' \setminus \overline{\Omega}$, \BBB $v_\eps \in W^{2,\infty}(\Omega' \setminus J_{v_\eps};\R^d)$, \EEE and  
\begin{align}\label{eq: dense-in-appli}
{\rm (i)} & \ \ v_\eps \to  u  \text{ in measure on } \Omega',\notag\\
{\rm (ii)} & \ \ \Vert e(v_\eps) - e(u) \Vert_{L^2(\Omega')} \to 0,\notag\\
{\rm (iii)} &  \ \  \mathcal{H}^{d-1}(J_{v_\eps})  \to  \mathcal{H}^{d-1}(J_{u}),\notag\\
{\rm (iv)} & \ \   \Vert \nabla v_\eps \Vert_{L^\infty(\Omega')} +  \Vert \nabla^2 v_\eps \Vert_{L^\infty(\Omega')} \le \eps^{(\beta-1)/2} \le \eps^{\gamma-1}.
\end{align}
Note that property (iv) can be achieved since the approximations satisfy $v_\eps \in W^{2,\infty}(\Omega' \setminus J_{v_\eps};\R^d)$. (Recall $\gamma < \beta < 1$.) Moreover,  $v_\eps \in W^{2,\infty}(\Omega' \setminus J_{v_\eps};\R^d)$ also implies $J_{\nabla v_\eps} \subset J_{v_\eps}$.

We define the sequence $y_\eps = \id + \eps v_\eps$. As $v_\eps \in GSBV_2^2(\Omega';\R^d)$ and  $v_\eps = h$ on $\Omega' \setminus \overline{\Omega}$, we get $y_\eps \in \mathcal{S}_{\eps,h}$, see \eqref{eq: boundary-spaces}. We now check that $y_\eps \rightsquigarrow u$ in the sense of Definition \ref{def:conv}.

We define $y_\eps^{\rm rot} = y_\eps$, i.e., the Caccioppoli partition in \eqref{eq: piecewise rotated} consists of the set $\Omega'$ only with corresponding rotation $\Id$. Then \eqref{eq: first conditions}(i),(ii) are trivially satisfied. As  $\nabla y_\eps^{\rm rot} - \Id = \eps \nabla v_\eps$, \eqref{eq: first conditions}(iii),(iv) follow from \eqref{eq: dense-in-appli}(ii),(iv). The rescaled displacement fields $u_\eps$ defined in \eqref{eq: rescalidipl} satisfy $u_\eps = v_\eps$. Then \eqref{eq: the main convergence}  for $E_u = \emptyset$ follows from \eqref{eq: dense-in-appli}(i)--(iii) and $J_{y_\eps} = J_{v_\eps}$.

Finally, we confirm $\lim_{\eps \to 0} \mathcal{E}_{\eps}(y_\eps) = \mathcal{E}(u)$. In view of $J_{y_\eps} = J_{v_\eps}$, $J_{\nabla y_\eps} \subset J_{y_\eps}$,  \eqref{eq: dense-in-appli}(iii), and the definition of the energies in \eqref{rig-eq: Griffith en},  \eqref{rig-eq: Griffith en-lim}, it suffices to show 
$$\lim_{\eps \to 0} \Big( \frac{1}{\eps^2}\int_{\Omega'} W(\nabla y_\eps)  + \frac{1}{\eps^{2\beta}} \int_{\Omega'} |\nabla^2 y_\eps|^2 \Big) = \int_{\Omega'} \frac{1}{2} Q(e(u)).$$
The second term vanishes by \eqref{eq: dense-in-appli}(iv), $\beta <1$, and the fact that $\nabla^2 y_\eps = \eps \nabla^2 v_\eps$. For the first term, we again use that $W(\Id + F) =  \frac{1}{2}Q( {\rm sym}(F)) + \omega(F)$ with $|\omega(F)|\le C|F|^3$ for $|F| \le 1$, and   compute by \eqref{eq: dense-in-appli}(ii),(iv)
\begin{align*} 
\lim_{\eps \to 0}\frac{1}{\eps^2} \int_{\Omega'} W(\nabla y_\eps) & = \lim_{\eps \to 0} \frac{1}{\eps^2} \int_{\Omega'} W(\Id +  \eps  \nabla v_\eps)  = \lim_{\eps \to 0}  \int_{\Omega'}  \Big( \frac{1}{2} Q(e(v_\eps)) + \frac{1}{\eps^2} \omega(\eps \nabla v_\eps)  \Big)  \\ & = \int_{\Omega'}   \frac{1}{2} Q(e(u)) +   \lim_{\eps \to 0}\int_{\Omega'}    {\rm O}\big( \eps|\nabla v_\eps|^3 \big)= \int_{\Omega'}   \frac{1}{2} Q(e(u)),
\end{align*}
where in the last step we have used that $\Vert \nabla v_\eps \Vert_{L^\infty(\Omega')} \le C\eps^{\gamma - 1}$ for some $\gamma >2/3$. This concludes the proof.  
\end{proof}

\begin{rem}\label{rem: Hoelder space}
The proof shows that one can readily incorporate a dependence on the material point $x$ in the density $W,$ as long as \eqref{assumptions-W} still holds. We also point out that it suffices to suppose that $W$ is $C^{2,\alpha}$ in a neighborhood of $SO(d)$, provided that $1 >\beta > \gamma > \frac{2}{2+\alpha}$. In fact, in that case, one has $|\omega(F)| \le C|F|^{2+\alpha}$ \RRR for all $|F|\le 1$,\EEE and all estimates remain true, where in \eqref{eq: eta chocie} one chooses $\eta_\eps$ with $\eps^{1-\gamma} \eta_\eps \to  +\infty$ and $\eps^\alpha \eta^{2+\alpha}_\eps \to 0$.
\end{rem} 

We close this subsection with the proof of Corollary \ref{cor: main cor}.

\begin{proof}[Proof of Corollary \ref{cor: main cor}]
The statement follows in the spirit of  the fundamental theorem of $\Gamma$-convergence, see, e.g., \cite[Theorem 1.21]{Braides:02}. We repeat the argument here for the reader's convenience. We observe that $\inf_{\bar{y} \in \mathcal{S}_{\eps,h}} \mathcal{E}_\eps(\bar{y})$ is uniformly bounded by choosing $\id + \eps h$ as competitor.  Given $(y_\eps)_\eps$, $y_\eps \in \mathcal{S}_{\eps,h}$, satisfying \eqref{eq: eps control}, we apply Theorem \ref{rig-th: gammaconv}(a) to find  a subsequence (not relabeled), and $u \in GSBD^2_h$ such that   $y_\eps \rightsquigarrow u$ in the sense of Definition \ref{def:conv}. Thus, by  Theorem   \ref{rig-th: gammaconv}(b) we obtain  \begin{align}\label{eq: last1}
 \mathcal{E}(u) \le  \liminf_{\eps \to 0} \mathcal{E}_\eps(y_\eps) \le  \liminf_{\eps \to 0} \inf_{\bar{y} \in \mathcal{S}_{\eps,h}} \mathcal{E}_\eps(\bar{y}).
 \end{align}
By  Theorem   \ref{rig-th: gammaconv}(c), for each $v \in GSBD^2_h$, there exists a   sequence $(w_\eps)_\eps$ with $w_\eps \rightsquigarrow v$ and  $\lim_{\eps\to 0} \mathcal{E}_\eps(w_\eps) = \mathcal{E}(v)$.  This  implies
  \begin{align}\label{eq: last2}
\limsup_{\eps \to 0}  \inf_{\bar{y} \in \mathcal{S}_{\eps,h}} \mathcal{E}_\eps(\bar{y}) \le \lim_{\eps  \to 0}  \mathcal{E}_\eps(w_\eps) =  \mathcal{E}(v).
  \end{align}
By combining \eqref{eq: last1}-\eqref{eq: last2} we find 
  \begin{align}\label{eq: last3}
 \mathcal{E}(u) \le \liminf_{\eps \to 0} \inf_{\bar{y} \in \mathcal{S}_{\eps,h}} \mathcal{E}_\eps(\bar{y}) \le  \limsup_{\eps \to 0} \inf_{\bar{y} \in \mathcal{S}_{\eps,h}} \mathcal{E}_\eps(\bar{y}) \le \mathcal{E}(v).
  \end{align}
  Since $v \in GSBD^2_h$ was arbitrary, we get that $u$ is a minimizer of $\mathcal{E}$. \BBB Property \eqref{eq: eps control2} follows from \EEE \eqref{eq: last3} with $v=u$. In particular, the limit in \eqref{eq: eps control2} does not depend on the specific choice of the subsequence and thus \eqref{eq: eps control2} holds for the whole sequence.
\end{proof}

\subsection{Characterization of limiting displacements}\label{sec: admissible}
  
This final subsection is devoted to the proof of  Lemma \ref{lemma: characteri}.

\begin{proof}[Proof of Lemma \ref{lemma: characteri}]
\emph{Proof of (a).} As a preparation, we observe that for two given rotations $R_1,R_2 \in SO(d)$ there holds
 \begin{align}\label{rig-eq: linearization-appli}
|{\rm sym}(R_2 R_1^T -\Id)|  \le C |R_1- R_2|^2.
\end{align}
This follows from formula \eqref{rig-eq: linearization} applied for $F = R_2 R_1^T$.  

Consider a sequence $(y_\eps)_\eps$. Let 
\begin{align}\label{eq: piecewise rotated2}
y^{{\rm rot},i}_\eps = \sum\nolimits_{j=1}^\infty {R_j^{\eps,i}} \, y_\eps  \, \chi_{P_j^{\eps,i}}, \ \ \ i=1,2,
\end{align}
be two sequences such that the corresponding rescaled displacement fields $u_\eps^i = \eps^{-1}(y_\eps^{{\rm rot},i} - \id)$, $i=1,2$, converge to $u_1$ and $u_2$, respectively, in the sense of \eqref{eq: the main convergence}, where the exceptional sets are denoted by $E_{u_1}$ and $E_{u_2}$, respectively. In particular, pointwise $\mathcal{L}^d$-a.e.\ in $\Omega'$ there holds
\begin{align}\label{eq:symmi}
e(u_\eps^1) -  e(u_\eps^2) &= \eps^{-1}   {\rm sym} \Big(     \sum\nolimits_j {R_j^{\eps,1}} \, \nabla y_\eps \, \chi_{P_j^{\eps,1}}   -  \sum\nolimits_j {R_j^{\eps,2}}  \, \nabla y_\eps   \, \chi_{P_j^{\eps,2}} \Big) \notag\\
&   = \eps^{-1}   {\rm sym} \Big(     \sum\nolimits_{j,k} \big( {R_j^{\eps,1}}  - {R_k^{\eps,2}} \big) \, \chi_{P_j^{\eps,1} \cap P_k^{\eps,2}} \, \nabla y_\eps   \Big)  \notag\\
&  = \eps^{-1}   {\rm sym} \Big(     \sum\nolimits_{j,k} \big( \Id  - {R_k^{\eps,2} (R_j^{\eps,1})^T } \big) \, \chi_{P_j^{\eps,1} \cap P_k^{\eps,2}} \, \nabla y^{{\rm rot},1}_\eps   \Big).
\end{align}  
 For brevity, we define $Z_\eps \in L^\infty(\Omega';\R^{d \times d})$ by  
\begin{align}\label{eq: Zdef}
Z_\eps :=  \sum\nolimits_{j,k} \big( \Id  - {R_k^{\eps,2} (R_j^{\eps,1})^T } \big) \, \chi_{P_j^{\eps,1} \cap P_k^{\eps,2}}. \end{align}
By \eqref{eq: first conditions}(iv) and the triangle inequality we get   
\begin{align*}
\sum_{j,k}  \big\|   {R_j^{\eps,1}}  - {R_k^{\eps,2} }\big\|^2_{L^2(P_j^{\eps,1} \cap P_k^{\eps,2})} &\le C\sum_{j=1}^\infty  \Vert (\nabla y_\eps)^T - R_j^{\eps,1}  \Vert^2_{L^2(P_j^{\eps,1})} + C\sum_{k=1}^\infty  \Vert (\nabla y_\eps)^T - R_k^{\eps,2}  \Vert^2_{L^2(P_k^{\eps,2})} \\
& = C\Vert \nabla y_\eps^{{\rm rot},1}- \Id \Vert_{L^2(\Omega')}^2 + C\Vert \nabla y_\eps^{{\rm rot},2}- \Id \Vert_{L^2(\Omega')}^2 \le C\eps^{2\gamma} 
\end{align*}
for some given $\gamma \in (\frac{2}{3},\beta)$, and $C>0$ independent of $\eps$.  Equivalently, this means  
\begin{align*}
 \sum\nolimits_{j,k} \mathcal{L}^d\big( P_j^{\eps,1} \cap P_k^{\eps,2} \big)  \big| R_j^{\eps,1}  - {R_k^{\eps,2}} \big|^2 \le C\eps^{2\gamma}.   
 \end{align*}
By recalling  \eqref{rig-eq: linearization-appli} and  \eqref{eq: Zdef}  we then get
\begin{align*}
\Vert {\rm sym} (Z_\eps) \Vert_{L^1(\Omega')} \le C\eps^{2\gamma}, \ \ \ \ \  \ \ \ \ \ \  \Vert Z_\eps \Vert_{L^2(\Omega')} \le C\eps^{\gamma}.
\end{align*}
This along with H\"older's inequality,  \eqref{eq: first conditions}(iv) for $y_\eps^{{\rm rot},1}$, and \eqref{eq:symmi}  yields
\begin{align}\label{eq: long eq}
\Vert e(u_\eps^1) -  e(u_\eps^2) \Vert_{L^1(\Omega')} & = \frac{1}{\eps}  \Vert  {\rm sym} \big( Z_\eps\, \nabla y_\eps^{{\rm rot},1} \big)\Vert_{L^1(\Omega')} \notag\\&
\le \frac{1}{\eps}    \Vert  {\rm sym} \big( Z_\eps\, \big(\nabla y_\eps^{{\rm rot},1} - \Id \big) \big)\Vert_{L^1(\Omega')} + \frac{1}{\eps}   \Vert  {\rm sym} \big( Z_\eps \big)\Vert_{L^1(\Omega')}\notag\\ &
\le \frac{1}{\eps}     \Vert Z_\eps \Vert_{L^2(\Omega')} \Vert  \nabla y_\eps^{{\rm rot},1} - \Id \Vert_{L^2(\Omega')} + \frac{1}{\eps}     \Vert  {\rm sym} \big( Z_\eps \big)\Vert_{L^1(\Omega')}  \le C\eps^{2\gamma - 1}.
\end{align}
We have that $e(u_\eps^1) -  e(u_\eps^2)$ converges to $e(u_1) - e(u_2)$ weakly  in $L^2(\Omega' \setminus (E_{u_1} \cup E_{u_2});\R^{d\times d}_{\rm sym})$, see \eqref{eq: the main convergence}(ii). Then \eqref{eq: long eq} and the fact that  $\gamma > \frac{2}{3} > \frac{1}{2}$ imply that $e(u_1) - e(u_2) = 0$ on $\Omega' \setminus (E_{u_1} \cup E_{u_2})$. This shows part (a) of the statement.  

\emph{Proof of (b).} Let $(y_\eps)_\eps$ be a sequence satisfying \eqref{eq: minimizer}.  Consider two piecewise rotated functions $y^{{\rm rot},i}_\eps$ as given in \eqref{eq: piecewise rotated2} and let $u_1,u_2$ be the limits identified in \eqref{eq: the main convergence}, where the corresponding exceptional sets are denoted by $E_{u_1},E_{u_2}$. We let $\mathcal{J}^i = \lbrace j \in \N: \,   P_j^{\eps,i} \subset \Omega \text{ up to an $\mathcal{L}^d$-negligible set} \rbrace$ for $i=1,2$, and set $D_\eps :=\bigcup_{i=1,2}\bigcup_{j \in \mathcal{J}^i} P_j^{\eps,i}$. By   \eqref{eq: first conditions}(ii) and $\gamma < \beta$ we \BBB obtain \EEE 
\begin{align}\label{eq: newD}
\limsup\nolimits_{\eps \to 0} \mathcal{H}^{d-1}\big(  \big(\partial^*D_\eps \cap \Omega' \big) \setminus \big(J_{y_\eps}\cup J_{\nabla y_\eps} \big)      \big) = 0.
\end{align}
As also $\sup_\eps\mathcal{H}^{d-1}(J_{y_\eps} \cup J_{\nabla y_{\eps}}) < + \infty$, we get   that $\mathcal{H}^{d-1}(\partial^* D_\eps)$ is uniformly controlled. Therefore, we may suppose that $D_\eps \to D$ in measure for a set of finite perimter $D \subset \Omega$, see \cite[Theorem 3.39]{Ambrosio-Fusco-Pallara:2000}. We observe that  $y^{{\rm rot},i}_\eps = y_\eps$ on $\Omega' \setminus D_\eps$ for $i=1,2$ by  Remark \ref{re: unchanged}. Therefore, \eqref{eq: rescalidipl}   implies that   $E_{u_1} \setminus D = E_{u_2} \setminus D$. In the following, we denote this set by $\hat{E}$. Then, \eqref{eq: rescalidipl} and \eqref{eq: the main convergence}(i)  also yield 
\begin{align}\label{eq: lateli}
u_1 = u_2 \ \ \ \text{ a.e.\ on } \Omega' \setminus (D \cup \hat{E}).
\end{align}
To compare $u_1$ and $u_2$ inside $D \cup \hat{E}$, we introduce modifications: for  $i=1,2$ and sequences $(\lambda_\eps)_\eps \subset \R^d$, let
\begin{align}\label{eq: defffi1}
y^{\lambda_\eps,i}_\eps := y^{{\rm rot},i}_\eps  + \lambda_\eps  \, \chi_{D_\eps}.
\end{align}
By definition, $D_\eps$ does not intersect $\Omega' \setminus \overline{\Omega}$ and has finite perimeter by \eqref{eq: newD}. Thus, we get $y^{\lambda_\eps,i}_\eps \in \mathcal{S}_{\eps,h}$, see \eqref{eq: boundary-spaces} \BBB and \eqref{eq: first conditions}(i). \EEE By \BBB \eqref{eq: first conditions}(ii), \EEE \eqref{eq: newD}, and the fact that the elastic energy is frame indifferent we also observe that $(y^{\lambda_\eps,i}_\eps)_\eps$ is  a  minimizing sequence for $i=1,2$ and all $(\lambda_\eps)_\eps \subset \R^d$. We obtain
\begin{align}\label{eq: the same}
y_\eps = \BBB y^{{\rm rot},i}_\eps = \EEE    y^{\lambda_\eps,i}_\eps \  \ \ \ \text{ on $\Omega' \setminus D_\eps $ for all $(\lambda_\eps)_\eps \subset \R^d$, $i=1,2$}.
\end{align}
This follows from \eqref{eq: defffi1} and $y^{{\rm rot},i}_\eps = y_\eps$ on $\Omega' \setminus D_\eps$ for $i=1,2$, see  Remark \ref{re: unchanged}. We now consider two different cases:

(1) Fix $i=1,2$, $\lambda \in \R^d$, and consider $\lambda_\eps = \lambda \eps$.  In view of  \eqref{eq: rescalidipl}, \eqref{eq: the main convergence}(i), and \eqref{eq: defffi1}, we get that $\eps^{-1}(y_\eps^{\lambda_\eps,i} - \id) \to  u_i + \lambda \chi_D $ in measure on $\Omega' \setminus E_{u_i}$. Thus, one can check that $y^{\lambda_\eps,i}_\eps\rightsquigarrow u^{\lambda}_i$   for some $u^\lambda_i \in GSBD^2_h$    satisfying
\begin{align}\label{eq: representation}
u^\lambda_i = u_i + \lambda \chi_D  \text{ on } \Omega'\setminus E_{u_i}.
\end{align}

(2) Recall that $\hat{E}= E_{u_1} \setminus D = E_{u_2} \setminus D = \lbrace x \in \Omega \setminus D: \,  |\eps^{-1}(y_\eps^{{\rm rot},i} - \id)| \to \infty\rbrace$ for $i=1,2$.  In view of \eqref{eq: defffi1}, we can choose a suitable sequence $(\lambda_\eps)_\eps$ such that $|\eps^{-1}(y_\eps^{\lambda_\eps,i} - \id)| \to \infty$ on $\hat{E} \cup D$ for $i=1,2$.  This along with \eqref{eq: the same} and \eqref{eq: the main convergence}(i),(iv) implies that for $i=1,2$ we have ${y}^{\lambda_\eps,i}_\eps \rightsquigarrow \hat{u}$ for some $\hat{u} \in GSBD_h^2$ satisfying  
\begin{align}\label{eq: hatti}
{\rm (i)} & \ \ \hat{u} = u_1 = u_2 \ \ \ \text{ a.e.\ on }  \Omega'\setminus (\hat{E} \cup D), \notag \\
{\rm (ii)} & \ \    e(\hat{u}) = 0 \ \ \text{ a.e.\ on } \ \hat{E} \cup D, \ \ \ \mathcal{H}^{d-1}(J_{\hat{u}} \cap (\hat{E} \cup D)^1) = 0,
\end{align}
 where $(\cdot)^1$ denotes the set of points with density $1$.

We now combine the cases  (1) and (2) to obtain the statement: since $(y_\eps^{\lambda_\eps,i})_\eps$ are minimizing sequences,  Corollary \ref{cor: main cor} implies that  each $u^\lambda_i$, $\lambda \in \R^d$, $i=1,2$,  and $\hat{u}$ are  minimizers of the problem $ \min_{v \in GSBD^2_h} \mathcal{E}(v)$. In particular, as $e(u_i^\lambda) = e(u_i)$ for   all $\lambda \in \R^d$ for both $i=1,2$, the jump sets of $u^\lambda_1$, $u^\lambda_2$ have to be independent of $\lambda$, i.e., $\mathcal{H}^{d-1}(J_{u_i} \triangle J_{u_i^\lambda}) = 0$ for all $\lambda \in \R^d$ and $i=1,2$.  In view of \eqref{eq: representation} and \eqref{eq: the main convergence}(iv), this yields $\partial^*E_{u_i} \cap \Omega', \partial^* (D \setminus E_{u_i}) \cap \Omega'   \subset J_{u_i} $  up to $\mathcal{H}^{d-1}$-negligigble sets. Since $\hat{E} = E_{u_i}\setminus D$, this implies  for $i=1,2$ that
\begin{align}\label{eq: jumpii}
\partial^* (\hat{E} \cup D)  \cap \Omega' \subset J_{u_i} \ \ \ \ \ \text{ up to $\mathcal{H}^{d-1}$-negligigble sets.}
\end{align}
Recall that $u_1,u_2$ are both minimizers, that also $\hat{u}$ is a minimzer, and that there holds $\hat{u} = u_1=u_2$ on $\Omega' \setminus (\hat{E} \cup D)$, see \eqref{eq: hatti}(i). This along with \eqref{eq: hatti}(ii) and \eqref{eq: jumpii} yields  $e({u}_i) = 0$ on $\hat{E} \cup D$ and $\mathcal{H}^{d-1}(J_{{u}_i} \cap (\hat{E} \cup D)^1) = 0$ for $i=1,2$. Then \eqref{eq: lateli} and \eqref{eq: jumpii} show that $e(u_1) = e(u_2)$ $\mathcal{L}^d$-a.e.\ on $\Omega'$, and $J_{u_1} = J_{u_2}$ up to an $\mathcal{H}^{d-1}$-negligible set.     
  \end{proof}

We finally provide an example that in case (a)  the strains cannot be compared inside $E_{u_1} \cup E_{u_2}$.

\begin{example}\label{e2}
{\normalfont
Similar to Example \ref{ex}, we consider $\Omega'= (0,3) \times (0,1)$, $\Omega = (1,3) \times (0,1)$, $\Omega_1 = (0,2) \times (0,1)$, $\Omega_2 = (2,3) \times (0,1)$, and $h \equiv 0$. Let $z \in W^{2,\infty}(\Omega';\R^d)$ with $\lbrace z = 0 \rbrace = \emptyset$, and define    
$$y_\eps(x) = x \chi_{\Omega_1}(x) + \big(x +  \eps z(x)\big) \chi_{\Omega_2}(x) \ \ \  \ \text{for} \ x \in \Omega'.$$
\RRR Note that $J_{y_\eps} = \partial\Omega_1 \cap\Omega' = \partial\Omega_2 \cap\Omega'$. \EEE Then two possible alternatives are
\begin{align*}
(1)& \ \ P^{\eps}_1 = \Omega_1, \  \ P^{\eps}_2 = \Omega_2, \ \  R_1^{\eps} = \Id, \ \ R_2^{\eps} = \bar{R}_\eps,\\
(2)& \ \ \tilde{P}^{\eps}_1 = \Omega',  \ \ \tilde{R}_1^{\eps} = \Id,
\end{align*}
where $\bar{R}_\eps \in SO(2)$ satisfies $\bar{R}_\eps = \Id + \eps^\gamma A + {\rm O}(\eps^{2\gamma}) $ for some $A \in \R^{2 \times 2}_{\rm skew}$, $\gamma \in (\frac{2}{3},\beta)$.   
Let $u_{\eps} = \eps^{-1} (\sum_{j=1}^2 R_j^{\eps}y_{\eps} \chi_{P_j^{\eps}} -\id)$ and $\tilde{u}_{\eps} = \eps^{-1}  (y_{\eps} -\id)$, We observe that $|u_\eps| \to \infty$ on $\Omega_2$. Possible limits identified in \eqref{eq: the main convergence} are $u = \lambda \chi_{\Omega_2}$ for some $\lambda \in \R^{d}$, $\lambda \neq 0$, with $E_u = \Omega_2$, and $\tilde{u}(x) = z(x) \,  \chi_{\Omega_2}(x)$ with $E_{\tilde{u}} = \emptyset$. This shows that in general there holds $e(u) \neq e(\tilde{u})$  in $E_{u}$.}
\end{example}

\section*{Acknowledgements} 
This work was supported by the DFG project FR 4083/1-1 and by the Deutsche Forschungsgemeinschaft (DFG, German Research Foundation) under Germany's Excellence Strategy EXC 2044 -390685587, Mathematics M\"unster: Dynamics--Geometry--Structure.


 \typeout{References}

\end{document}